\documentclass[reqno,12pt,letterpaper]{amsart}
\usepackage[proof]{sdmacros}

\title[Semiclassical measures on hyperbolic surfaces]%
{Semiclassical measures on hyperbolic surfaces\\
have full support}
\author{Semyon Dyatlov}
\email{dyatlov@math.berkeley.edu}
\address{Department of Mathematics, Massachusetts Institute of Technology,
77 Massachusetts Ave, Cambridge, MA 02139}
\address{Department of Mathematics, University of California, Berkeley, CA 94720}
\author{Long Jin}
\email{ljin@math.tsinghua.edu.cn}
\address{Yau Mathematical Sciences Center, Tsinghua University, Beijing, China}
\address{Department of Mathematics, Purdue University,
150 N. University St, West Lafayette, IN 47907}

\begin{document}

\begin{abstract}
We show that each limiting semiclassical measure obtained from a sequence of eigenfunctions of the Laplacian
on a compact hyperbolic surface is supported on the entire cosphere bundle.
The key new ingredient for the proof is the fractal uncertainty principle, first formulated in~\cite{hgap} and proved for porous sets in~\cite{fullgap}.
\end{abstract}

\maketitle

\addtocounter{section}{1}
\addcontentsline{toc}{section}{1. Introduction}

Let $(M,g)$ be a compact (connected) hyperbolic surface, that is a Riemannian
surface of constant curvature $-1$.
Denote by $\Delta$ the (nonpositive) Laplace--Beltrami operator.
We fix a semiclassical quantization procedure (see~\S\ref{s:prelim-opers})
$$
a\in C_0^\infty(T^*M)\ \mapsto\ \Op_h(a):L^2(M)\to L^2(M),\quad
h>0.
$$
Assume that $u_j$ is a sequence of  eigenfunctions of $-\Delta$
with eigenvalues $h_j^{-2}\to \infty$:
\begin{equation}
  \label{e:u-j}
(-h_j^2\Delta-I)u_j=0,\quad
\|u_j\|_{L^2}=1,\quad
h_j>0,\quad
h_j\to 0\quad\text{as }j\to\infty.
\end{equation}
We say that $u_j$ converge semiclassically to some probability measure $\mu$ on $T^*M$
if
$$
\langle \Op_{h_j}(a)u_j,u_j\rangle_{L^2}\to \int_{T^*M} a\,d\mu\quad\text{as }j\to\infty
\text{ for all }a\in C_0^\infty(T^*M).
$$
We say $\mu$ is a \emph{semiclassical defect measure} (or in short, semiclassical measure) if $\mu$ is the semiclassical limit
of some sequence of eigenfunctions.
It is well-known (see for instance~\cite[\S\S5.1,5.2]{e-z}) that
each semiclassical defect measure is supported on the cosphere bundle $S^*M\subset T^*M$
and it is invariant under the geodesic flow $\varphi_t:S^*M\to S^*M$. However
not every invariant measure can be a semiclassical defect measure as follows
from our first result: 
\begin{theo}
  \label{t:basic}
Let $\mu$ be a semiclassical defect measure. Then $\supp \mu=S^*M$,
that is for every nonempty open set $\mathcal U\subset S^*M$
we have $\mu(\mathcal U)>0$.
\end{theo}
If $a\in C^\infty(M)$ depends only on $x$, then $\Op_h(a)$ is the multiplication
operator by~$a$. Therefore Theorem~\ref{t:basic} implies that the support of any weak limit of the measures $|u_j|^2\,d\vol_g$ (often called \emph{quantum limit}) is equal to $M$.

The \emph{quantum ergodicity} theorem of Shnirelman, Zelditch, and Colin de Verdi\`ere~\cite{Shnirelman,ZelditchQE,CdV} (see also Helffer--Martinez--Robert and Zelditch--Zworski~\cite{HMR,ZelditchZworski}
for more general versions)
implies that there is a density one sequence of eigenvalues of $\Delta$ such that
the corresponding eigenfunctions converge weakly to the Liouville measure $\mu_L$.
The \emph{quantum unique ergodicity (QUE)} conjecture of Rudnick and Sarnak~\cite{RudnickSarnak}
states that $\mu_L$ is the only semiclassical measure. This conjecture was
proved for Hecke forms on arithmetic surfaces (such as the modular surface)
by Lindenstrauss and Soundararajan~\cite{Lindenstrauss,Sound}.
For the related setting of Eisenstein series see Luo--Sarnak and Jakobson~\cite{LuoSarnak,JakobsonQUE}.
For the history of the QUE conjecture we refer the reader to the reviews of 
Marklof~\cite{MarklofReview}, Zelditch~\cite{ZelditchReview}, and Sarnak~\cite{SarnakQUE}.

In the more general setting of manifolds with Anosov geodesic flows,
restrictions on possible semiclassical measures have been obtained
by Anantharaman and Anantha\-raman--Nonnenmacher~\cite{Anantharaman,AnantharamanNonnenmacher};
see also Rivi\`ere~\cite{Riviere1,Riviere2}
and Anantha\-raman--Silberman~\cite{AnantharamanSilberman}.
In particular, \cite[Theorem~1.2]{AnantharamanNonnenmacher} shows
that every semiclassical measure on a hyperbolic surface
has Kolmogorov--Sinai entropy $\geq 1/2$. For comparison,
the Liouville measure has entropy 1 and the delta
measure on a closed geodesic has entropy 0.
Examples of manifolds with
ergodic but non-Anosov geodesic flows with quasimodes and eigenfunctions
which violate QUE have been constructed by Donnelly~\cite{Donnelly} and
Hassell~\cite{Hassell}; see also Faure--Nonnenmacher--de Bi\`evre~\cite{FNB}.

Theorem~\ref{t:basic} is in some sense orthogonal to the entropy bounds
discussed above.
For instance, Theorem~\ref{t:basic} excludes the case of
$\mu$ supported on a set of dimension $3-\varepsilon$,
which might have entropy very close to~1.
On the other hand, it does not exclude the case
$\mu=\alpha \mu_L+(1-\alpha)\mu_0$, where $\mu_0$ is a delta measure
on a closed geodesic and $0<\alpha\leq 1$,
while the entropy bound excludes such measures with $\alpha<1/2$.
Theorem~\ref{t:basic} also does not exclude
the case when $\mu$ is a countable linear combination
of the measures~$\delta_{\gamma_k}$
where $\{\gamma_k\}_{k=1}^\infty$ are all the closed geodesics:
for instance, $\mu=\sum_{k=1}^\infty2^{-k}\delta_{\gamma_k}$ satisfies
$\supp\mu=S^*M$.

Our second result is a more quantitative version of Theorem~\ref{t:basic}:
\begin{theo}
  \label{t:estimate}
Assume that $a\in C_0^\infty(T^*M)$ and $a|_{S^*M}\not\equiv 0$.
Then there exist constants $\mathbf C(a),\mathbf h_0(a)>0$ depending only on $M,a$
such that for $0<h<\mathbf h_0(a)$ and all $u\in H^2(M)$
\begin{equation}
  \label{e:estimate}
\|u\|_{L^2}\leq 
\mathbf C(a)\|\Op_h(a)u\|_{L^2}+{\mathbf C(a) \log(1/h)\over h}\big\|(-h^2\Delta-I)u\big\|_{L^2}.
\end{equation}
\end{theo}
Theorem~\ref{t:basic} follows immediately from Theorem~\ref{t:estimate}.
Indeed, take $a\in C_0^\infty(T^*M)$ such that $a|_{S^*M}\not\equiv 0$
but $\supp a\cap S^*M\subset \mathcal U$. Let $u_j$, $h_j$ satisfy~\eqref{e:u-j}.
Then~\eqref{e:estimate} implies that
$\|\Op_{h_j}(a)u_j\|_{L^2}\geq \mathbf C(a)^{-1}$ for large $j$.
However, if $u_j$ converge semiclassically to some measure $\mu$, then
$$
\|\Op_{h_j}(a)u_j\|_{L^2}^2\to \int_{T^*M}|a|^2\,d\mu\quad\text{as }j\to\infty.
$$
It follows that $\int |a|^2\,d\mu>0$ and thus $\mu(\mathcal U)>0$.

The above argument shows that Theorem~\ref{t:basic} still holds
if we replace the requirement
$(-h_j^2\Delta-I)u_j=0$ in~\eqref{e:u-j}
by $\|(-h_j^2\Delta-I)u_j\|_{L^2}=o(h_j/\log(1/h_j))$,
that is it applies to $o(h/\log(1/h))$ quasimodes.
This quasimode strength is almost sharp;
indeed, Brooks, Eswarathasan--Nonnenmacher, and Eswarathasan--Silberman~\cite{Brooks,Suresh1,Suresh2} construct
a family of $\mathcal O(h/\log(1/h))$ quasimodes
which do not converge to $\mu_L$. In particular, \cite[Proposition~1.9]{Suresh1}
gives $\mathcal O(h/\log(1/h))$ quasimodes which
converge semiclassically to the delta measure on any given closed geodesic.
We remark that the factor $h^{-1}\log(1/h)$ in~\eqref{e:estimate}
is reminiscent of the scattering resolvent bounds on the real line
for mild hyperbolic trapping, see~\cite[\S 3.2]{ZworskiReview} and the references there.

Theorem~\ref{t:estimate} has applications to control for the Schr\"odinger equation~\cite{JinControl} and its proof can be adapted to show exponential energy decay for the damped wave equation~\cite{JinDWE}.

We would also like to mention
a recent result of Logunov--Malinnikova~\cite{Remez}
giving a bound of the following form for an eigenfunction $u$,
$(-h^2\Delta-I)u=0$:
\begin{equation}
  \label{e:loma-1}
\sup_\Omega |u|\geq C^{-1} \big(\vol_g(\Omega)/C\big)^{-C/h}
\cdot \sup_M |u|
\end{equation}
where $C$ is a constant depending only on $M$. The bound~\eqref{e:loma-1}
holds on any closed Riemannian manifold and for any subset $\Omega\subset M$ of positive volume.
For hyperbolic surfaces and $\Omega$ having nonempty interior, Theorem~\ref{t:estimate}
together with the unique continuation principle give the bound
\begin{equation}
  \label{e:loma-2}
\|u\|_{L^2(\Omega)}\geq c_\Omega \|u\|_{L^2(M)}
\end{equation}
where $c_\Omega>0$ is a constant depending on $M,\Omega$ but not on $h$.
Unlike~\eqref{e:loma-1}, the bound~\eqref{e:loma-2} cannot hold for general Riemannian manifolds:
if $M$ is the round sphere and $\Omega$ lies strictly inside one hemisphere, then
there exists a sequence of Gaussian beam eigenfunctions $u$ concentrating on the equator
with $\|u\|_{L^2(\Omega)}\leq e^{-C/h}\|u\|_{L^2(M)}$.

\subsection{Outline of the proof}

We give a rough outline of the proof of Theorem~\ref{t:estimate}, assuming for simplicity that
$(-h^2\Delta-I)u=0$. We write
$$
u=A_{\mathcal X}u+A_{\mathcal Y}u
$$
where $A_{\mathcal X}$, $A_{\mathcal Y}$ are constructed from
two fixed pseudodifferential operators $A_1,A_2$ conjugated by the wave propagator
for times up to $2\rho\log(1/h)$,
see~\eqref{e:A-subset} and~\eqref{e:XY-def}.
The parameter $\rho$ is chosen less than~1 but is very close to~1, see
the remark following Proposition~\ref{l:ultimate-fup}.
The operators $A_{\mathcal X},A_{\mathcal Y}$ formally correspond to symbols
$a_{\mathcal X},a_{\mathcal Y}$ such that for some small parameter $\alpha>0$
\begin{itemize}
\item for $(x,\xi)\in\supp a_{\mathcal X}$, at most
$2\alpha\log(1/h)$ of the points
\begin{equation}
  \label{e:rooby}
\varphi_j(x,\xi),\quad j=0,1,\dots,2\rho\log(1/h)
\end{equation}
lie in $\{a\neq 0\}$. That is, the geodesic $\varphi_t(x,\xi)$,
$0\leq t\leq 2\rho\log(1/h)$ spends very little time in $\{a\neq 0\}$;
\item for $(x,\xi)\in\supp a_{\mathcal Y}$, at least
${1\over 10}\alpha\log(1/h)$ points~\eqref{e:rooby}
lie in $\{a\neq 0\}$.
\end{itemize}
To explain the intuition behind the argument, we
first consider the case when $\alpha=0$, that is
for $(x,\xi)\in \supp a_{\mathcal X}$
none of the points~\eqref{e:rooby} lie in $\{a\neq 0$\}.
(In the argument for general $\alpha$ leading to~\eqref{e:esta-2}, putting $\alpha=0$
is equivalent to taking $\alpha\sim 1/\log(1/h)$.)
One can view $\{a\neq 0\}$ as a `hole' in $S^*M$
and $\supp a_{\mathcal X}$ is contained in the set of `forward trapped'
geodesics (that is, those that do not go through the hole).
On the other hand, points $(x,\xi)$ in $\supp a_{\mathcal Y}$
are \emph{controlled} in the sense that $\varphi_j(x,\xi)$ lies in the hole 
for some $j\in [0,2\rho\log(1/h)]$. Therefore one hopes to control
$A_{\mathcal Y}u$ in terms of $\Op_h(a)u$ using Egorov's theorem and the fact
that $u$ is an eigenfunction of the Laplacian~-- see~\eqref{e:esta-2.0} below.

The operator
$A_{\mathcal X}$ is not pseudodifferential
because it corresponds to propagation for time $2\rho\log(1/h)$ which
is much larger than the Ehrenfest time~$\log(1/h)$. However, conjugating
$A_{\mathcal X}$ by the wave group we obtain a product
of the form $\mathcal A_-\mathcal A_+$ where
the symbols $a_\pm$ corresponding to $\mathcal A_\pm$
satisfy
$$
\varphi_{\mp j}(\supp a_\pm)\cap \{a\neq 0\}=\emptyset\quad\text{for all }
j=0,1,\dots,\rho\log(1/h).
$$
That is, $\supp a_-$ is `forward trapped'
and $\supp a_+$ is `backward trapped'.
The operators $\mathcal A_\pm$ lie in the calculi associated
to the weak unstable/stable Lagrangian foliations on $T^*M\setminus 0$
similar to the ones developed by Dyatlov--Zahl~\cite{hgap},
see~\S\ref{s:fancy-calculus} and the Appendix.
More precisely, the symbol $a_+$ is regular along the weak unstable
foliation and $a_-$ is regular along the weak stable foliation.
The constant curvature condition plays an important role in defining these calculi associated to Lagrangian foliations. On a general surface with negative curvature, the weak unstable/stable Lagrangian foliations are only H\"{o}lder continuous instead of smooth.

Using unique ergodicity of horocyclic flows due to Furstenberg~\cite{Furstenberg} we show that
$\supp a_+$ is porous in the stable direction and $\supp a_-$ is porous in the unstable direction
(see Definition~\ref{d:porous-TM} and Lemma~\ref{l:porosity-verified}).
Then the fractal uncertainty principle of Bourgain--Dyatlov~\cite{hgap} implies
that $\|\mathcal A_-\mathcal A_+\|_{L^2\to L^2}\leq Ch^\beta$
for some $\beta>0$ and thus (see Proposition~\ref{l:ultimate-fup})
\begin{equation}
  \label{e:esta-1}
\|A_{\mathcal X}u\|_{L^2}\leq Ch^\beta\|u\|_{L^2}.
\end{equation}
We stress that just like the operator $A_{\mathcal X}$, the product $\mathcal A_-\mathcal A_+$
is not a pseudodifferential operator since it corresponds to propagation for time $\rho\log(1/h)>{1\over 2}\log(1/h)$ in \emph{both} time directions.
(In fact, if
$\mathcal A_-\mathcal A_+$ were pseudodifferential with symbol $a_-a_+$, we would expect the left-hand
side of~\eqref{e:esta-1} to be asymptotic to $\sup |a_-a_+|=1$.)
However since $\rho<1$ each of the operators $\mathcal A_\pm$,
corresponding to propagation for time $\rho\log(1/h)$ in \emph{one} time direction, is still
pseudodifferential in an anisotropic class, see~\S\ref{s:fancy-calculus}
(but the product $\mathcal A_-\mathcal A_+$ is not pseudodifferential
since the calculi in which $\mathcal A_-$ and $\mathcal A_+$ lie are incompatible with each other).
The norm estimate~\eqref{e:esta-1} uses fractal uncertainty principle, which is a tool from harmonic analysis,
and in some sense goes beyond the classical/quantum correspondence.

To estimate $A_{\mathcal Y}u$ in the case $\alpha=0$,
we can break it into pieces, each of which
corresponds to the condition $\varphi_j(x,\xi)\in \{a\neq 0\}$
for some $j=0,1,\dots,2\rho\log(1/h)$.
Since $(-h^2\Delta-I)u=0$, $u$ is equivariant under the
wave propagator; therefore, each piece can be controlled by $\Op_h(a)u$.
Summing over $j$, we get
\begin{equation}
  \label{e:esta-2.0}
\|A_{\mathcal Y}u\|_{L^2}\leq C\log(1/h)\|\Op_h(a)u\|_{L^2}
+\mathcal O(h^\infty)\|u\|_{L^2}.
\end{equation}

Combining~\eqref{e:esta-1} and~\eqref{e:esta-2.0}
we get~\eqref{e:estimate}, however the term
$\|\Op_h(a)u\|_{L^2}$ comes with an extra factor of $\log(1/h)$.
To remove this factor, we take $\alpha$ small, but positive.
The estimate~\eqref{e:esta-1} still holds
as long as $\alpha$ is chosen small enough depending on the
fractal uncertainty exponent~$\beta$, see~\eqref{e:X-totaled}.
Moreover, we get the following
improved version of~\eqref{e:esta-2.0} for some $\varepsilon>0$
(see Proposition~\ref{l:nontrapped-estimate};
one can take $\varepsilon=1/8$)
\begin{equation}
  \label{e:esta-2}
\|A_{\mathcal Y}u\|_{L^2}\leq {C\over\alpha}\|\Op_h(a)u\|_{L^2}+\mathcal O(h^\varepsilon)\|u\|_{L^2}.
\end{equation}
Combining~\eqref{e:esta-1} and~\eqref{e:esta-2} gives
the required bound~\eqref{e:estimate}.

The estimate~\eqref{e:esta-2} is delicate because $A_{\mathcal Y}$ is not pseudodifferential.
To prove it, we adapt some of the methods of~\cite{Anantharaman}.
More precisely, if we replace $2\rho\log(1/h)$ by $\tilde\varepsilon\log(1/h)$ for small enough
$\tilde \varepsilon>0$ in the definition of $A_{\mathcal Y}$, then $A_{\mathcal Y}$ is pseudodifferential in a mildly exotic calculus
and one can use a semiclassical version of the Chebyshev inequality
(see Lemma~\ref{l:Y-control-1}) to establish~\eqref{e:esta-2}. To
pass from short logarithmic times to time $2\rho\log(1/h)$, we use
a submultiplicative estimate, see the end of~\S\ref{s:nontrapped-proof}.

\section{Preliminaries}
  \label{s:prelims}

\subsection{Dynamics of geodesic and horocyclic flows}

Let $(M,g)$ be a compact hyperbolic surface and $T^*M\setminus 0$
consist of elements of the cotangent bundle $(x,\xi)\in T^*M$ such that $\xi\neq 0$.
Denote by $S^*M=\{|\xi|_g=1\}$ the cosphere bundle.
Define the symbol $p\in C^\infty(T^*M\setminus 0;\mathbb R)$ by
\begin{equation}
  \label{e:p-def}
p(x,\xi)=|\xi|_g.
\end{equation}
The Hamiltonian flow of $p$,
\begin{equation}
  \label{e:p-flow}
\varphi_t:=\exp(tH_p):T^*M\setminus 0\to T^*M\setminus 0
\end{equation}
is the homogeneous geodesic flow.

Henceforth we assume that $M$ is orientable; if not, we may pass to a double cover of $M$.
We use an explicit frame on $T^*M\setminus 0$ consisting of four vector fields
\begin{equation}
  \label{e:canonical-fields}
H_p,U_+,U_-,D
\in C^\infty\big(T^*M\setminus 0;T(T^*M\setminus 0)\big).
\end{equation}
Here $H_p$ is the generator of $\varphi_t$ and $D=\xi\cdot\partial_\xi$
is the generator of dilations. The vector fields $U_\pm$
are defined on $S^*M$ as stable ($U_+$)
and unstable ($U_-$) horocyclic vector fields and extended homogeneously
to $T^*M\setminus 0$, so that
\begin{equation}
  \label{e:comm-rel-0}
[U_\pm,D]=[H_p,D]=0.
\end{equation}
See for instance~\cite[(2.1)]{rrh}. The vector fields $U_\pm$
are tangent to the level sets of $p$ and
satisfy the commutation relations
\begin{equation}
  \label{e:comm-rel}
[H_p,U_\pm]=\pm U_\pm.
\end{equation}
Thus on each level set of $p$, the flow $\varphi_t$ has a flow/stable/unstable decomposition,
with $U_+$ spanning the stable space and $U_-$ spanning the unstable space;
see for instance~\cite[(3.14)]{rrh}.
We use the following notation for the weak stable/unstable spaces:
\begin{equation}
  \label{e:l-foliations}
L_s:=\Span(H_p,U_+),\quad
L_u:=\Span(H_p,U_-)\
\subset\
T(T^*M\setminus 0).
\end{equation}
Then $L_s,L_u$ are Lagrangian foliations, see~\cite[Lemma~4.1]{hgap}.

The next statement, used in~\S\ref{s:pf-fup} to establish
the porosity condition, is a consequence of the unique ergodicity of horocyclic flows,
see~\cite{Furstenberg,Marcus,Ratner,Coudene,Hubbard}.
\begin{prop}
  \label{l:horocycle-unique}
Let $\mathcal U\subset S^*M$ be a nonempty open set. Then there exists
$T>0$ depending only on $M,\mathcal U$ such that for all $(x,\xi)\in S^*M$,
\begin{equation}
\label{e:recurrence}
\{e^{s U_\pm}(x,\xi)\mid 0\leq s\leq T\}\cap \mathcal U\neq\emptyset.
\end{equation}
\end{prop}
\begin{proof}
We focus on the case of $U_+$; the same proof applies to $U_-$.
Denote by $\mu_L$ the Liouville probability measure on $S^*M$. 
By the unique ergodicity of the horocyclic flow~$e^{sU_+}$,
$\mu_L$ is the only probability measure on $S^*M$
invariant under $e^{sU_+}$.

Let $f\in C(S^*M)$ be a continuous function. Then we have uniform convergence
\begin{equation}
  \label{e:horun}
\langle f\rangle_T:={1\over T}\int_0^T f\circ e^{sU_+}\,ds\to \langle f\rangle_\mu:=\int_{S^*M}f\,d\mu_L\quad\text{as }T\to \infty.
\end{equation}
Indeed, assume that~\eqref{e:horun} is false. Then there exists
$\varepsilon>0$ and
sequences $T_k\to\infty$, $(x_k,\xi_k)\in S^*M$ such that
\begin{equation}
  \label{e:horun2}
\big|\langle f\rangle_{T_k}(x_k,\xi_k)-\langle f\rangle_\mu\big|\geq\varepsilon.
\end{equation}
Consider the probability measures $\nu_k$ on $S^*M$ defined by
$$
\int_{S^*M}g\,d\nu_k=\langle g\rangle_{T_k}(x_k,\xi_k)\quad\text{for all }g\in C(S^*M).
$$
Passing to a subsequence, we may assume that $\nu_k$ converge weakly
to some probability measure $\nu$. Since $T_k\to\infty$, the measure
$\nu$ is invariant under the flow $e^{sU_+}$, thus $\nu=\mu_L$.
However, $\int f\,d\nu\neq \int f\,d\mu_L$ by~\eqref{e:horun2}, giving a contradiction.
This finishes the proof of~\eqref{e:horun}.

Now, choose $f\in C(S^*M)$ such that
$$
\supp f\subset\mathcal U,\quad
\langle f\rangle_\mu=1.
$$
By~\eqref{e:horun}, there exists $T>0$ such that
$\langle f\rangle_T>1/2$ everywhere. This implies~\eqref{e:recurrence}.
\end{proof}

\subsection{Operators and propagation}
  \label{s:prelim-opers}

We use the standard classes of semiclassical pseudodifferential operators
with classical symbols $\Psi^k_h(M)$, with $\Psi^{\comp}_h(M)$ denoting operators $A\in\Psi^k_h(M)$
such that the wavefront set $\WFh(A)$ is a compact subset of $T^*M$.
We refer the reader to the book of Zworski~\cite{e-z}
for an introduction to semiclassical analysis used in this paper, to~\cite[\S14.2.2]{e-z}
for pseudodifferential operators on manifolds,
and to~\cite[\S E.1.5]{dizzy} and~\cite[\S2.1]{hgap} for the classes $\Psi^k_h(M)$ used here.
Denote by $S^k(T^*M)$ the corresponding symbol classes,
and by
$$
\sigma_h:\Psi^k_h(M)\to S^k(T^*M),\quad
\Op_h:S^k(T^*M)\to \Psi^k_h(M)
$$
the principal symbol map and a (non-canonical) quantization map.
For $A,B\in\Psi^k_h(M)$ and an open set $U\subset T^*M$, we say that
$A=B+\mathcal O(h^\infty)$ \emph{microlocally on} $U$,
if $\WFh(A-B)\cap U=\emptyset$.

We have the following norm bound:
\begin{equation}
  \label{e:basic-norm-bound}
A\in\Psi^0_h(M),\quad
\sup|\sigma_h(A)|\leq 1
\quad\Longrightarrow\quad
\|A\|_{L^2\to L^2}\leq 1+Ch.
\end{equation}
Indeed, applying the
sharp G\r arding inequality~\cite[Theorem~4.32]{e-z} to
the operator $I-A^*A$ we get for all $u\in L^2(M)$
$$
\|u\|_{L^2}^2-\|Au\|_{L^2}^2=
\langle (I-A^*A)u,u\rangle_{L^2}\geq -Ch\|u\|_{L^2}^2
$$
which gives~\eqref{e:basic-norm-bound}.

The operator $-h^2\Delta$ lies in $\Psi^2_h(M)$ and, with $p$ defined in~\eqref{e:p-def},
$$
\sigma_h(-h^2\Delta)=p^2.
$$
For us it will be convenient to have an operator with principal symbol
$p$, since the corresponding Hamiltonian flow is homogeneous. Of course,
we have to cut away from the zero section as $p$ is not smooth there.
We thus fix a function
$$
\psi_P\in C_0^\infty((0,\infty);\mathbb R),\quad
\psi_P(\lambda)=\sqrt{\lambda}\quad\text{for }{1\over 16}\leq \lambda\leq 16,
$$
and define the operator
\begin{equation}
  \label{e:the-P}
P:=\psi_P(-h^2\Delta),\quad
P^*=P.
\end{equation}
By the functional calculus of pseudodifferential operators,
see~\cite[Theorem~14.9]{e-z} or~\cite[\S8]{DimassiSjostrand}, we have
\begin{equation}
  \label{e:the-P-2}
P\in\Psi^{\comp}_h(M),\quad
\sigma_h(P)=p\quad\text{on }\{1/4\leq |\xi|_g\leq 4\}.
\end{equation}
To quantize the flow $\varphi_t$, we use the propagator
\begin{equation}
  \label{e:U-t}
U(t):=\exp\Big(-{itP\over h}\Big):L^2(M)\to L^2(M).
\end{equation}
The operator $U(t)$ is unitary on $L^2(M)$.

For a bounded operator $A:L^2(M)\to L^2(M)$, define
\begin{equation}
  \label{e:A-t}
A(t):=U(-t)A U(t).
\end{equation}
If $A\in \Psi^{\comp}_h(M)$, $\WFh(A)\subset \{1/4<|\xi|_g<4\}$,
and $t$ is bounded uniformly in $h$,
then Egorov's theorem~\cite[Theorem~11.1]{e-z}
implies that
\begin{equation}
  \label{e:basic-egorov}
A(t)\in \Psi^{\comp}_h(M);\quad
\sigma_h(A(t))=\sigma_h(A)\circ\varphi_t.
\end{equation}

\subsection{Anisotropic calculi and long time propagation}
  \label{s:fancy-calculus}

If $A\in\Psi^{\comp}_h(M)$ and $t$ grows with $h$ then $A(t)$ will generally not be pseudodifferential
in the class $\Psi^{\comp}_h$
since the derivatives of the symbol $\sigma_h(A)\circ\varphi_t$
may grow exponentially with $t$.
In this section we introduce a more general calculus which contains
the operators $A(t)$ for $|t|\leq \rho\log(1/h)$, $\rho<1$.
(More precisely, we will have two calculi, one of which works for $t\geq 0$
and the other, for $t\leq 0$.)
Our calculus is similar to the one developed in~\cite[\S3]{hgap},
with remarks on the differences of these two calculi
and the proofs of some of the properties of the calculus contained the Appendix.

Fix $\rho\in [0,1)$ and let $L\in \{L_u,L_s\}$ where the Lagrangian foliations $L_u,L_s$
are defined in~\eqref{e:l-foliations}. Define the class of $h$-dependent symbols
$S^{\comp}_{L,\rho}(T^*M\setminus 0)$ as follows:
$a\in S^{\comp}_{L,\rho}(T^*M\setminus 0)$ if
\begin{enumerate}
\item $a(x,\xi;h)$ is smooth in $(x,\xi)\in T^*M\setminus 0$,
defined for $0<h\leq 1$,
and supported in an $h$-independent compact
subset of $T^*M\setminus 0$;
\item $\sup_{x,\xi}|a(x,\xi;h)|\leq C$ for some constant $C$ and all $h$;
\item $a$ satisfies the derivative bounds
\begin{equation}
  \label{e:symbol-derby}
\sup_{x,\xi}|Y_1\ldots Y_mZ_1\ldots Z_k a(x,\xi;h)|\leq Ch^{-\rho k-\varepsilon},\quad
0<h\leq 1
\end{equation}
for all $\varepsilon>0$ and all vector fields
$Y_1,\dots,Y_m,Z_1,\dots,Z_k$ on $T^*M\setminus 0$ such that
$Y_1,\dots,Y_m$ are tangent to $L$.
Here the constant $C$ depends on $Y_1,\dots,Y_m$, $Z_1,\dots,Z_k$, and~$\varepsilon$
but does not depend on $h$.
\end{enumerate}
This class is slightly larger than the one in~\cite[Definition~3.2]{hgap}
because  we require~\eqref{e:symbol-derby}
to hold for all $\varepsilon>0$, while~\cite{hgap} had $\varepsilon:=0$.

We use the following notation:
$$
f(h)=\mathcal O(h^{\alpha-})\quad\text{if }f(h)=\mathcal O(h^{\alpha-\varepsilon})\text{ for all }\varepsilon>0.
$$
In terms of the frame~\eqref{e:canonical-fields}, the derivative
bounds~\eqref{e:symbol-derby} become
\begin{align}
  \label{e:derb-s}
\sup_{x,\xi} \big|H_p^k U_+^\ell U_-^m D^n a(x,\xi;h)|&=\mathcal O(h^{-\rho(m+n)-})\quad\text{for }L=L_s,\\
  \label{e:derb-u}
\sup_{x,\xi} \big|H_p^k U_-^\ell U_+^m D^n a(x,\xi;h)|&=\mathcal O(h^{-\rho(m+n)-})\quad\text{for }L=L_u.
\end{align}
If $a\in C_0^\infty(T^*M\setminus 0)$ is an $h$-independent symbol, then it follows
from the commutation relations~\eqref{e:comm-rel-0} and~\eqref{e:comm-rel} that
$$
H_p^k U_+^\ell U_-^m D^n(a\circ\varphi_t)=e^{(m-\ell)t}(H_p^k U_+^\ell U_-^m D^na)\circ\varphi_t.
$$
Therefore
\begin{equation}
  \label{e:derbound-long-stable}
a\circ\varphi_t\in S^{\comp}_{L_s,\rho}(T^*M\setminus 0)\quad\text{uniformly in }
t,\quad
0\leq t\leq \rho\log(1/h).
\end{equation}
Similarly
\begin{equation}
  \label{e:derbound-long-unstable}
a\circ\varphi_{-t}\in S^{\comp}_{L_u,\rho}(T^*M\setminus 0)\quad\text{uniformly in }
t,\quad
0\leq t\leq \rho\log(1/h).
\end{equation}
Let $\Psi^{\comp}_{h,L,\rho}(T^*M\setminus 0)$, $L\in \{L_u,L_s\}$, be the classes 
of pseudodifferential operators with symbols in $S^{\comp}_{L,\rho}$
defined following the same construction as in~\cite[\S3]{hgap}.
They satisfy similar properties to the operators used in~\cite{hgap},
in particular they are pseudolocal and bounded on $L^2(M)$
uniformly in $h$. However, the $\mathcal O(h^{1-\rho})$ remainders
have to be replaced by $\mathcal O(h^{1-\rho-})$ because of the relaxed
assumptions on derivatives~\eqref{e:symbol-derby}.
We denote by
$$
\Op_h^L:a\in S^{\comp}_{L,\rho}(T^*M\setminus 0)\ \mapsto\ \Op_h^L(a)\in \Psi^{\comp}_{h,L,\rho}(T^*M\setminus 0)
$$
a (non-canonical) quantization procedure. See~\S\ref{s:appendix-general} for more details.

The $\Psi^{\comp}_{h,L,\rho}$ calculus satisfies a version of Egorov's Theorem,
Proposition~\ref{l:egorov-long}. It states that for $A=\Op_h(a)$
where $a\in C_0^\infty(\{1/4<|\xi|_g<4\})$ is independent of~$h$,
\begin{align}
  \label{e:long-egorov-1}
A(t)&=\Op_h^{L_s}(a\circ\varphi_t)+\mathcal O(h^{1-\rho-})_{L^2\to L^2},\\
  \label{e:long-egorov-2}
A(-t)&=\Op_h^{L_u}(a\circ\varphi_{-t})+\mathcal O(h^{1-\rho-})_{L^2\to L^2}
\end{align}
uniformly in $t\in [0,\rho\log(1/h)]$. 

\section{Proof of Theorem~\ref{t:estimate}}
  \label{s:proof}

In this section we give the proof of Theorem~\ref{t:estimate}.
It uses two key estimates, Proposition~\ref{l:nontrapped-estimate}
and Proposition~\ref{l:ultimate-fup}, which
are proved in~\S\ref{s:nontrapped} and~\S\ref{s:fup} respectively.

\subsection{Partitions and words}
  \label{s:partition}

We assume that $a\in C_0^\infty(T^*M)$ and $a|_{S^*M}\not\equiv 0$
as in the assumptions of Theorem~\ref{t:estimate}. Fix conic open sets
$$
\mathcal U_1,\mathcal U_2\subset T^*M\setminus 0,\quad
\mathcal U_1,\mathcal U_2\neq \emptyset,\quad
\overline{\mathcal U_1}\cap\overline{\mathcal U_2}=\emptyset,\quad
\overline{\mathcal U_2}\cap S^*M\subset \{a\neq 0\}.
$$
(The sets $\mathcal U_j$ and the conditions~\eqref{e:A-12-2} below are used
in the proof of Proposition~\ref{l:ultimate-fup}.)

We introduce a pseudodifferential partition of unity
$$
I=A_0+A_1+A_2,\quad
A_0\in\Psi^0_h(M),\quad
A_1,A_2\in\Psi^{\comp}_h(M)
$$
such that (see Figure~\ref{f:partition}):
\begin{itemize}
\item $A_0$ is microlocalized away from the cosphere bundle $S^*M$.
More specifically, we put $A_0:=\psi_0(-h^2\Delta)$ where
$\psi_0\in C^\infty(\mathbb R;[0,1])$ satisfies
$$
\supp\psi_0\cap [1/4,4]=\emptyset,\quad
\supp(1-\psi_0)\subset (1/16, 16).
$$
This implies that
$$
\WFh(A_0)\cap \{1/2\leq |\xi|_g\leq 2\}=\emptyset,\quad
\WFh(I-A_0)\subset \{1/4<|\xi|_g<4\}.
$$
\item $A_1,A_2$ are microlocalized in an energy shell and away from $\mathcal U_1,\mathcal U_2$, that is
\begin{gather}
  \label{e:A-12-1}
\WFh(A_1)\cup\WFh(A_2)\subset \{1/4<|\xi|_g <4\},\\
  \label{e:A-12-2}
\WFh(A_1)\cap \overline{\mathcal U_1}=\WFh(A_2)\cap\overline{\mathcal U_2}=\emptyset.
\end{gather}
\item $A_1$ is controlled by $a$ on the cosphere bundle,
that is
\begin{equation}
  \label{e:A-1-a}
\WFh(A_1)\cap S^*M\subset \{a\neq 0\}.
\end{equation}
\end{itemize}
To construct $A_1,A_2$, note that~\eqref{e:A-12-1}--\eqref{e:A-1-a}
are equivalent to $\WFh(A_j)\subset \Omega_j$ where
$$
\begin{aligned}
\Omega_1&:=\big(\{1/4<|\xi|_g<4\}\setminus\overline{\mathcal U_1}\big)
\cap\big(
\{a\neq 0\}\cup (T^*M\setminus S^*M)
\big),\\
\Omega_2&:=\{1/4<|\xi|_g<4\}\setminus\overline{\mathcal U_2}
\end{aligned}
$$
are open subsets of $T^*M$ such that
$\WFh(I-A_0)\subset \{1/4<|\xi|_g<4\}\subset \Omega_1\cup\Omega_2$.
It remains to use a pseudodifferential partition of unity
to find $A_1,A_2$ such that~\eqref{e:A-12-1}--\eqref{e:A-1-a} hold
and $A_1+A_2=I-A_0$.
(For instance, one can write $I-A_0=\Op_h(b)+\mathcal O(h^\infty)$
where $\supp b\subset \Omega_1\cup\Omega_2$, split $b=a_1+a_2$ for some
symbols $a_1,a_2$ with $\supp a_j\subset\Omega_j$, and put $A_j:=\Op_h(a_j)$.)
We moreover choose $A_1,A_2$ so that
\begin{figure}
\qquad\includegraphics{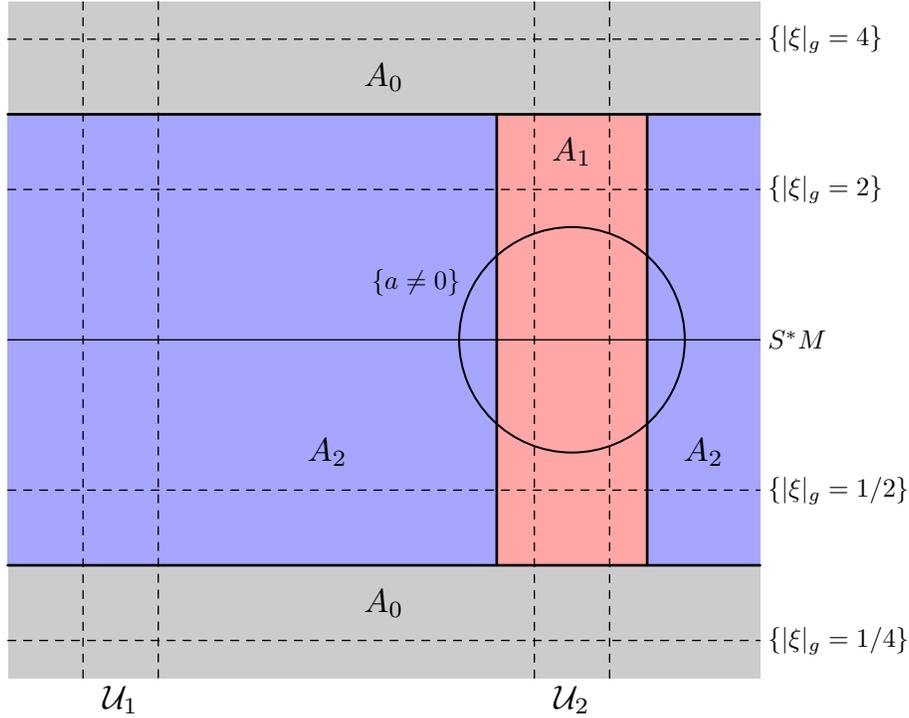}
\caption{The sets $\mathcal U_1,\mathcal U_2$,
$\WFh(A_j)$ (shaded),
and $\{a\neq 0\}$ inside $T^*M$. The vertical
direction corresponds to dilating $\xi$.}
\label{f:partition}
\end{figure}
\begin{equation}
  \label{e:A-symbol}
0\leq a_\ell\leq 1\quad\text{where }
a_\ell:=\sigma_h(A_\ell),\quad
\ell=0,1,2.
\end{equation}
We next dynamically refine the partition $A_j$. For each $n\in\mathbb N_0$,
define the set of words of length $n$,
$$
\mathcal W(n):=\{1,2\}^n=\big\{\mathbf w=w_0\dots w_{n-1}\mid
w_0,\dots,w_{n-1}\in \{1,2\}\big\}.
$$
For each word $\mathbf w=w_0\dots w_{n-1}\in \mathcal W(n)$, using the notation~\eqref{e:A-t} define the operator
\begin{equation}
	\label{e:op-word}
A_{\mathbf w}=A_{w_{n-1}}(n-1) A_{w_{n-2}}(n-2)\cdots A_{w_1}(1) A_{w_0}(0).
\end{equation}
If $n$ is bounded independently of $h$, then by Egorov's Theorem~\eqref{e:basic-egorov}
we have $A_{\mathbf w}\in \Psi^{\comp}_h(M)$ and
$\sigma_h(A_{\mathbf w})=a_{\mathbf w}$ where
\begin{equation}
	\label{e:symbol-word}
a_{\mathbf w}=\prod_{j=0}^{n-1} \big(a_{w_j}\circ\varphi_j\big).
\end{equation}
For a subset $\mathcal E\subset \mathcal W(n)$, define
the operator $A_{\mathcal E}$ and the symbol $a_{\mathcal E}$ by
\begin{equation}
  \label{e:A-subset}
A_{\mathcal E}:=\sum_{\mathbf w\in \mathcal E}A_{\mathbf w},\quad
a_{\mathcal E}:=\sum_{\mathbf w\in \mathcal E}a_{\mathbf w}.
\end{equation}
Since $A_1+A_2=I-A_0$ and $P$ are both functions of $\Delta$,
they commute with each other.
Therefore, $A_1+A_2$ commutes with $U(t)$ which implies
\begin{equation}
  \label{e:A-total}
A_{\mathcal W(n)}=(A_1+A_2)^n.
\end{equation}
This operator is equal to the identity microlocally near $S^*M$,
implying
\begin{lemm}
  \label{l:A-0}
We have for all $n\geq 0$ and $u\in H^2(M)$,
\begin{equation}
  \label{e:A-0}
\|u-(A_1+A_2)^n u\|_{L^2}\leq C\|(-h^2\Delta-I)u\|_{L^2}.
\end{equation}
\end{lemm}
\begin{proof}
Since $A_1+A_2=I-A_0=I-\psi_0(-h^2\Delta)$ we have
$$
u-(A_1+A_2)^n u= \psi_1(-h^2\Delta) (-h^2\Delta-I)u,\quad
\psi_1(\lambda):={1-(1-\psi_0(\lambda))^n\over \lambda-1}.
$$
Since $1\notin\supp\psi_0$ we have $\sup_{\lambda\in\mathbb R}|\psi_1(\lambda)|\leq C$
for some constant $C$ independent of $n$, and~\eqref{e:A-0} follows.
\end{proof}

\subsection{Long words and key estimates}

Take $\rho\in (0,1)$ very close to 1, to be chosen later (in Proposition~\ref{l:ultimate-fup}),
and put
$$
N_0:=\Big\lceil{\rho\over 4}\log(1/h)\Big\rceil\in\mathbb N,\quad
N_1:=4N_0\approx \rho\log(1/h).
$$
Then words of length $N_0$ and $N_1$ give rise to pseudodifferential operators
in the calculus $\Psi^{\comp}_{h,L,\rho}$ discussed in~\S\ref{s:fancy-calculus}:
\begin{lemm}
  \label{l:long-word-egorov}
For each $\mathbf w\in \mathcal W(N_0)$ we have (with bounds independent of $\mathbf w$)
\begin{equation}
  \label{e:lwe-1}
a_{\mathbf w}\in S^{\comp}_{L_s,\rho/4}(T^*M\setminus 0),\quad
A_{\mathbf w}=\Op_h^{L_s}(a_{\mathbf w})+\mathcal O(h^{3/4})_{L^2\to L^2}.
\end{equation}
If instead $\mathbf w\in \mathcal W(N_1)$, then
\begin{equation}
  \label{e:lwe-2}
a_{\mathbf w}\in S^{\comp}_{L_s,\rho}(T^*M\setminus 0),\quad
A_{\mathbf w}=\Op_h^{L_s}(a_{\mathbf w})+\mathcal O(h^{1-\rho-})_{L^2\to L^2}.
\end{equation}
\end{lemm}
\begin{proof}
We prove~\eqref{e:lwe-2}; the proof of~\eqref{e:lwe-1} is identical, replacing
$\rho$ by $\rho/4$.
First of all, by~\eqref{e:derbound-long-stable} and~\eqref{e:A-symbol} we have uniformly in $j=0,\dots,N_1-1$
\begin{equation}
  \label{e:lweint-1}
a_{w_j}\circ\varphi_j\in S^{\comp}_{L_s,\rho}(T^*M\setminus 0),\quad
\sup|a_{w_j}\circ\varphi_j|\leq 1.
\end{equation}
Recalling the definition~\eqref{e:symbol-word}, we have $a_{\mathbf w}\in S^{\comp}_{L_s,\rho}(T^*M\setminus 0)$
by Lemma~\ref{l:long-product-symbols},
where we put $a_j:=a_{w_j}\circ\varphi_j$.
Here we use the relation~\eqref{e:funny-epsilons}
of the classes $S^{\comp}_{L_s,\rho,\rho'}$ used in the Appendix
to the class $S^{\comp}_{L_s,\rho}$ used here.
Next, by Lemma~\ref{l:egorov-long} we have
uniformly in $j=0,\dots,N_1-1$
\begin{equation}
  \label{e:lweint-2}
A_{w_j}(j)=\Op_h^{L_s}(a_{w_j}\circ\varphi_j)+\mathcal O(h^{1-\rho-})_{L^2\to L^2}.
\end{equation}
Applying Lemma~\ref{l:long-product-operators} with $A_j:=A_{w_j}(j)$, we get
$A_{\mathbf w}=\Op_h^{L_s}(a_{\mathbf w})+\mathcal O(h^{1-\rho-})_{L^2\to L^2}$.
\end{proof}
Now, define the density function
\begin{equation}
	\label{e:fn-density}
F:\mathcal W(N_0)\to\mathbb [0,1],\quad
F(w_0\dots w_{N_0-1})={\#\{j\in \{0,\dots,N_0-1\}\mid w_j=1\}\over N_0}.
\end{equation}
Fix small $\alpha\in (0,1)$ to be chosen later (in~\eqref{e:the-alpha})
and define
\begin{equation}
	\label{e:Z-control-set}
\mathcal Z:=\{F\geq \alpha\}\subset\mathcal W(N_0).
\end{equation}
We call words $\mathbf w\in\mathcal Z$ \emph{controlled} because for
each $(x,\xi)\in\supp a_{\mathbf w}$,
at least $\alpha N_0$ of the points $\varphi_0(x,\xi),\varphi_1(x,\xi),\dots,\varphi_{N_0-1}(x,\xi)$
lie in $\supp a_1$ and due to~\eqref{e:A-1-a} are controlled by $a$.

We chose $N_0$ short enough so that the operators $A_{\mathbf w}$, $\mathbf w\in \mathcal W(N_0)$
are pseudodifferential and Egorov's Theorem~\eqref{e:lwe-1} holds with remainder
$\mathcal O(h^{3/4})$. This will be convenient for the estimates in~\S\ref{s:nontrapped}
below, in particular in Lemma~\ref{l:small-time-good} (explaining why we did not replace
$N_0$ with $N_1$).
However, to apply the fractal uncertainty principle
(Proposition~\ref{l:ultimate-fup}), we need to propagate for time $2N_1=8N_0\approx 2\rho\log(1/h)$.
To bridge the resulting gap, we define the set of controlled words $\mathcal Y\subset \mathcal W(2N_1)$
by iterating~$\mathcal Z$.
More specifically, writing words in $\mathcal W(2N_1)$
as concatenations $\mathbf w^{(1)}\dots \mathbf w^{(8)}$
where $\mathbf w^{(1)},\dots,\mathbf w^{(8)}\in \mathcal W(N_0)$,
define the partition
\begin{equation}
  \label{e:XY-def}
\begin{aligned}
\mathcal W(2N_1)&=\mathcal X\sqcup\mathcal Y,\\
\mathcal X&:=\{\mathbf w^{(1)}\dots \mathbf w^{(8)}\mid \mathbf w^{(\ell)}\notin\mathcal Z\quad\text{for all }\ell\},\\
\mathcal Y&:=\{\mathbf w^{(1)}\dots \mathbf w^{(8)}\mid \text{there exists $\ell$ such that }\mathbf w^{(\ell)}\in\mathcal Z\}
\end{aligned}
\end{equation}
In our argument the parameter $\alpha$ will be taken small so that
$\mathcal X$ has few elements. The size of $\mathcal X$ is estimated by
the following statement (which is not sharp but provides a bound sufficient
for us)
\begin{lemm}
  \label{l:X-count}
The number of elements in $\mathcal X$ is bounded by (here $C$ may depend on $\alpha$)
\begin{equation}
  \label{e:X-count}
\#(\mathcal X)\leq Ch^{-4\sqrt{\alpha}}.
\end{equation}
\end{lemm}
\begin{proof}
The complement $\mathcal W(N_0)\setminus\mathcal Z$ consists of words
$\mathbf w=w_0\dots w_{N_0-1}$, $w_j\in \{1,2\}$,
such that the set $S_{\mathbf w}=\{j\mid w_j=1\}$ has no more than $\lfloor\alpha N_0\rfloor$ elements.
We add arbitrary elements to the set $S_{\mathbf w}$ to ensure it has size exactly $\lfloor \alpha N_0\rfloor$.
Each choice of~$S_{\mathbf w}$ corresponds to at most
$2^{\alpha N_0}\leq h^{-\alpha/4}$ words $\mathbf w$,
and by Stirling's formula
$$
\#\{S_{\mathbf w}\mid \mathbf w\in \mathcal W(N_0)\setminus \mathcal Z\}
\leq\binom{N_0}{\lfloor \alpha N_0\rfloor}
\leq C\exp\big(-(\alpha\log\alpha+(1-\alpha)\log(1-\alpha))N_0\big).
$$
Since $-(\alpha\log\alpha+(1-\alpha)\log(1-\alpha))\leq \sqrt{\alpha}$
for $0\leq\alpha\leq 1$ we have
$$
\#(\mathcal W(N_0)\setminus\mathcal Z)\leq Ch^{-\alpha/4-\sqrt{\alpha}/4}\leq Ch^{-\sqrt{\alpha}/2}.
$$
Since $\#(\mathcal X)=\#(\mathcal W(N_0)\setminus\mathcal Z)^8$,
we obtain~\eqref{e:X-count}.
\end{proof}
Now we state the two key estimates used in the proof. The first one, proved
in~\S\ref{s:nontrapped}, estimates
the mass of an approximate eigenfunction on the controlled region $\mathcal Y$:
\begin{prop}
\label{l:nontrapped-estimate}
We have for all $u\in H^2(M)$, with $A_{\mathcal Y}$ defined by~\eqref{e:A-subset}
\begin{equation}
  \label{e:nontrapped-estimate}
\|A_{\mathcal Y}u\|_{L^2}\leq {C\over\alpha}\|\Op_h(a)u\|_{L^2}+{C\log(1/h)\over \alpha h}\|(-h^2\Delta-I)u\|_{L^2}
+\mathcal O(h^{1/8})\|u\|_{L^2}
\end{equation}
where the constant $C$ does not depend on $\alpha$.
\end{prop}
The second estimate, proved in~\S\ref{s:fup} using a fractal uncertainty principle,
is a norm bound on the operator corresponding to every single word of length $2N_1\approx 2\rho\log(1/h)$:
\begin{prop}
\label{l:ultimate-fup}
There exist $\beta>0$, $\rho\in (0,1)$ depending only on $M,\mathcal U_1,\mathcal U_2$ such that
$$
\sup_{\mathbf w\in\mathcal W(2N_1)}
\|A_{\mathbf w}\|_{L^2\to L^2}\leq Ch^{\beta}.
$$
\end{prop}
\Remark
Since the proof of~\cite[Proposition~4.2]{fullgap} uses the triangle inequality, the estimate
on the norm of $A_{\mathbf w}$ is $\mathcal O(h^{\tilde\beta-2(1-\rho)})$ for some $\tilde\beta>0$
depending on $M,\mathcal U_1,\mathcal U_2$,
thus $\rho$ has to be close enough to~1 depending on $\tilde\beta$ to get decay of this norm.
On the other hand we cannot put $\rho=1$ since the calculus described in~\S\ref{s:fancy-calculus} only
works for~$\rho<1$.

\subsection{End of the proof of Theorem~\ref{t:estimate}}
  \label{s:end-of-proof}

Take $\beta,\rho$ from Proposition~\ref{l:ultimate-fup};
we may assume that $\beta<1/8$.
Since $A_{\mathcal X}+A_{\mathcal Y}=A_{\mathcal W(2N_1)}=(A_1+A_2)^{2N_1}$ by~\eqref{e:A-total}, we have
for all $u\in H^2(M)$
$$
\|u\|_{L^2}\leq \|A_{\mathcal X}u\|_{L^2}
+\|A_{\mathcal Y}u\|_{L^2}
+\|u-(A_1+A_2)^{2N_1}u\|_{L^2}.
$$
Combining Lemma~\ref{l:X-count} with Proposition~\ref{l:ultimate-fup}
and using the triangle inequality, we have
\begin{equation}
  \label{e:X-totaled}
\|A_{\mathcal X}u\|_{L^2}=\mathcal O(h^{\beta-4\sqrt{\alpha}})\|u\|_{L^2}.
\end{equation}
Combining this with Proposition~\ref{l:nontrapped-estimate}
and Lemma~\ref{l:A-0}
we obtain
\begin{equation}
  \label{e:almost-there}
\|u\|_{L^2}\leq {C\over\alpha}\|\Op_h(a)u\|_{L^2}+{C\log(1/h)\over\alpha h}\|(-h^2\Delta-I)u\|_{L^2}
+\mathcal O(h^{\beta-4\sqrt{\alpha}})\|u\|_{L^2}.
\end{equation}
Choosing 
\begin{equation}
  \label{e:the-alpha}
\alpha:={\beta^2\over 64},\quad
\beta-4\sqrt{\alpha}={\beta\over 2}
\end{equation}
and taking $h$ small enough 
to remove the $\mathcal O(h^{\beta/2})$ term on the right-hand side of~\eqref{e:almost-there},
we obtain~\eqref{e:estimate}, finishing the proof.

\section{The controlled region}
  \label{s:nontrapped}

In this section we prove Proposition~\ref{l:nontrapped-estimate},
estimating an approximate eigenfunction $u$ on geodesics which spend a positive
fraction of their time inside $\{a\neq 0\}$. The proof uses tools similar to~\cite[\S2]{Anantharaman}.

\subsection{Control and propagation}

Recall the operator $A_1\in\Psi^{\comp}_h(M)$ constructed in~\S\ref{s:partition}.
We first use the wavefront set restriction~\eqref{e:A-1-a}
to estimate $A_1u$:
\begin{lemm}
  \label{l:basic-control-1}
We have for all $u\in H^2(M)$
\begin{equation}
  \label{e:bc1}
\|A_1u\|_{L^2}\leq C\|\Op_h(a)u\|_{L^2}+C\|(-h^2\Delta-I)u\|_{L^2}+Ch\|u\|_{L^2}.
\end{equation}
\end{lemm}
\begin{proof}
By~\eqref{e:A-1-a} we have $\supp a_1\cap S^*M\subset \{a\neq 0\}$
where $a_1=\sigma_h(A_1)$. Since $p^2-1$ is a defining
function for $S^*M$,
there exist $b,q\in C_0^\infty(T^*M)$ such that
$a_1=ab+q(p^2-1)$. It follows that
\begin{equation}
  \label{e:bc1int}
A_1=\Op_h(b)\Op_h(a)+\Op_h(q)(-h^2\Delta-I)+\mathcal O(h)_{L^2\to L^2}.
\end{equation}
It remains to apply~\eqref{e:bc1int} to $u$ and use
the fact that $\Op_h(b),\Op_h(q)$ are bounded on $L^2$ uniformly in $h$.
\end{proof}
Next, if we control $Au$ for some operator $A$, then we also control
$A(t)u$ where $A(t)$ is defined using~\eqref{e:A-t}:
\begin{lemm}
  \label{l:propagated}
Assume that $A:L^2(M)\to L^2(M)$ is bounded uniformly in $h$.
Then there exists a constant $C$ such that for all $t\in\mathbb R$ and $u\in H^2(M)$
\begin{equation}
  \label{e:propagated}
\|A(t)u\|_{L^2}\leq \|Au\|_{L^2}+{C|t|\over h}\|(-h^2\Delta-I)u\|_{L^2}.
\end{equation}
\end{lemm}
\begin{proof}
Recall from~\eqref{e:A-t} that $A(t)=U(-t)AU(t)$
where $U(t)=\exp(-itP/h)$ and $P\in\Psi^{\comp}_h(M)$ is defined in~\eqref{e:the-P}.
Since
$$
\partial_t\big(e^{it/h}U(t)\big)=-{i\over h}e^{it/h}U(t)(P-I),
$$
integrating from $0$ to $t$ we have
$$
\|U(t)u-e^{-it/h}u\|_{L^2}
=\|e^{it/h}U(t)u-u\|_{L^2}\leq {|t|\over h}\|(P-I)u\|_{L^2}.
$$
Then
\begin{equation}
  \label{e:propagated-int-1}
\|A(t)u\|_{L^2}=\|A U(t)u\|_{L^2}
\leq \|A u\|_{L^2}+{C|t|\over h}\|(P-I)u\|_{L^2}.
\end{equation}
We have $P-I=\psi_E(-h^2\Delta)(-h^2\Delta-I)$
where $\psi_E(\lambda)=(\psi_P(\lambda)-1)/(\lambda-1)$.
Therefore
\begin{equation}
  \label{e:propagated-int-2}
\|(P-I)u\|_{L^2}\leq C\|(-h^2\Delta-I)u\|_{L^2}.
\end{equation}
Combining~\eqref{e:propagated-int-1} and~\eqref{e:propagated-int-2} we obtain~\eqref{e:propagated}.
\end{proof}
Combining Lemmas~\ref{l:basic-control-1} and~\ref{l:propagated}, we obtain
\begin{lemm}
  \label{l:propagated-control}
For all $t\in \mathbb R$ and $u\in H^2(M)$, we have
\begin{equation}
  \label{e:propagatedc}
\|A_1(t)u\|_{L^2}\leq C\|\Op_h(a) u\|_{L^2}
+{C\langle t\rangle\over h}\|(-h^2\Delta-I)u\|_{L^2}
+Ch\|u\|_{L^2}
\end{equation}
where $\langle t\rangle :=\sqrt{1+t^2}$ and
the constant $C$ is independent of $t$ and $h$.
\end{lemm}

\subsection{Operators corresponding to weighted words}

By Lemma~\ref{l:long-word-egorov}, for each $\mathbf w\in \mathcal W(N_0)$
the operator $A_{\mathbf w}$ is pseudodifferential
modulo an $\mathcal O(h^{3/4})_{L^2\to L^2}$ remainder.
However, for a subset $\mathcal E\subset \mathcal W(N_0)$
the operator $A_{\mathcal E}$ defined in~\eqref{e:A-subset}
is the sum of many operators of the form $A_{\mathbf w}$
and thus a priori might not even be bounded on $L^2$ uniformly in~$h$.
In this section we show that $A_{\mathcal E}$ is still a pseudodifferential
operator plus a small remainder, using the fact that
the corresponding symbol $a_{\mathcal E}$ is bounded.

More generally one can consider operators obtained by assigning a coefficient to each word.
For a function $c:\mathcal W(N_0)\to\mathbb C$, define the operator $A_c$ and the symbol $a_c$ by
\begin{equation}
	\label{e:op-symbol-weight}
A_{c}:=\sum_{\mathbf w\in \mathcal W(N_0)}c(\mathbf w)A_{\mathbf w},
\quad
a_c:=\sum_{\mathbf w\in\mathcal W(N_0)}c(\mathbf w)a_{\mathbf w}.
\end{equation}
Note that for $\mathcal E\subset \mathcal W(N_0)$ we have
$A_{\mathcal E}=A_{\mathbf 1_{\mathcal E}}$ where
$\mathbf 1_{\mathcal E}$ is the indicator function of $\mathcal E$.

The next lemma shows that the operator $A_c$ is pseudodifferential
modulo a small remainder. Recall the symbol classes $S^{\comp}_{L_s,\rho,\rho'}(T^*M\setminus 0)$
introduced in~\S\ref{s:appendix-symbols}.
\begin{lemm}
  \label{l:small-time-good}
Assume $\sup|c|\leq 1$. Then 
\begin{equation}
  \label{e:stg}
a_c\in S^{\comp}_{L_s,1/2,1/4}(T^*M\setminus 0),\quad
A_c=\Op_h^{L_s}(a_c)+\mathcal O(h^{1/2})_{L^2\to L^2}.
\end{equation}
The $S^{\comp}_{L_s,1/2,1/4}$ seminorms of $a_c$ and
the constant in $\mathcal O(h^{1/2})$ are independent of $c$.
\end{lemm}
\begin{proof}
We first show that $a_c\in S^{\comp}_{L_s,1/2,1/4}(T^*M\setminus 0)$.
Since $a_1,a_2\geq 0$ and $a_1+a_2=1-a_0\leq 1$, we have for all $(x,\xi)\in T^*M\setminus 0$
$$
|a_c(x,\xi)|\leq a_{\mathcal W(N_0)}(x,\xi)=\prod_{j=0}^{N_0-1}(a_1+a_2)(\varphi_j(x,\xi))\leq 1.
$$
It remains to show that for $m+k>0$ and all vector fields $Y_1,\dots,Y_m,Z_1,\dots,Z_k$ on $T^*M\setminus 0$
such that $Y_1,\dots,Y_m$ are tangent to $L_s$ we have
\begin{equation}
  \label{e:stg-1}
\sup|Y_1\dots Y_mZ_1\dots Z_k a_c|\leq Ch^{-k/2-m/4}. 
\end{equation}
By the triangle inequality the left-hand side of~\eqref{e:stg-1} is bounded by
$$
\sum_{\mathbf w\in\mathcal W(N_0)}\sup|Y_1\dots Y_mZ_1\dots Z_k a_{\mathbf w}|.
$$
By~\eqref{e:lwe-1} each summand is bounded by $Ch^{-k/4-0.01}$ where $C$ is independent of $\mathbf w$.
The number of summands is equal to $2^{N_0}\leq h^{-1/4+0.01}$.
Therefore the left-hand side of~\eqref{e:stg-1} is bounded
by $Ch^{-(k+1)/4}\leq Ch^{-k/2-m/4}$, giving~\eqref{e:stg-1}.

Finally, by~\eqref{e:lwe-1} we have
$$
A_c=\sum_{\mathbf w\in \mathcal W(N_0)}c(\mathbf w)\big(\Op_h^{L_s}(a_{\mathbf w})+\mathcal O(h^{3/4})_{L^2\to L^2}\big)
=\Op_h^{L_s}(a_c)+\mathcal O(h^{1/2})_{L^2\to L^2}
$$
finishing the proof.
\end{proof}
Combining Lemma~\ref{l:small-time-good} with the sharp G\r arding inequality (Lemma~\ref{l:gaarding})
we deduce the following ``almost monotonicity'' property for norms of the operators $A_c$:
\begin{lemm}
  \label{l:almost-positivity}
Assume $c,d:\mathcal W(N_0)\to \mathbb R$ and
$|c(\mathbf w)|\leq d(\mathbf w)\leq 1$ for all $\mathbf w\in \mathcal W(N_0)$.
Then for all $u\in L^2(M)$ we have
$$
\|A_c u\|_{L^2}\leq \|A_d u\|_{L^2}+ Ch^{1/8}\|u\|_{L^2}
$$
where the constant $C$ is independent of $c,d$.
\end{lemm}
\begin{proof}
By~\eqref{e:stg} we may replace $A_c,A_d$ by $\Op_h^{L_s}(a_c),\Op_h^{L_s}(a_d)$. It is then enough to prove
$$
\|\Op_h^{L_s}(a_c)u\|_{L^2}^2\leq \|\Op_h^{L_s}(a_d)u\|_{L^2}^2+Ch^{1/4}\|u\|_{L^2}^2.
$$
This is equivalent to
\begin{equation}
  \label{e:apos-1}
\langle Bu,u\rangle_{L^2}\geq -Ch^{1/4}\|u\|_{L^2}^2,\quad
B:=\Op_h^{L_s}(a_d)^*\Op_h^{L_s}(a_d)
-\Op_h^{L_s}(a_c)^*\Op_h^{L_s}(a_c). 
\end{equation}
Recall that $a_c,a_d\in S^{\comp}_{L_s,1/2,1/4}(T^*M\setminus 0)$.
By~\eqref{e:genius-4} and~\eqref{e:genius-5}
we have
\begin{equation}
  \label{e:apos-2}
B=\Op_h^{L_s}(a_d^2-a_c^2)+\mathcal O(h^{1/4})_{L^2\to L^2}.
\end{equation}
Since $|c(\mathbf w)|\leq d(\mathbf w)$ for all $\mathbf w$,
we have $0\leq a_d^2-a_c^2\in S^{\comp}_{L_s,1/2,1/4}(T^*M\setminus 0)$. Then by Lemma~\ref{l:gaarding}
\begin{equation}
  \label{e:apos-3}
\Re\langle \Op_h^{L_s}(a_d^2-a_c^2)u,u\rangle_{L^2}\geq -Ch^{1/4}\|u\|_{L^2}^2.
\end{equation}
Combining~\eqref{e:apos-2} and~\eqref{e:apos-3}, we get~\eqref{e:apos-1}, finishing the proof.
\end{proof}

\subsection{Proof of Proposition \ref{l:nontrapped-estimate}}
  \label{s:nontrapped-proof}

We first estimate $A_{\mathcal Z}u$ where $\mathcal{Z}\subset\mathcal{W}(N_0)$ is the set of controlled words
defined in~\eqref{e:Z-control-set}:
\begin{lemm}
  \label{l:Y-control-1}
We have for all $u\in H^2(M)$, with the constant $C$ independent of $\alpha$
\begin{equation}
  \label{e:Yc-1}
\|A_{\mathcal Z}u\|_{L^2}\leq
{C\over\alpha}\|\Op_h(a)u\|_{L^2}+{C\log(1/h)\over \alpha h}\|(-h^2\Delta-I)u\|_{L^2}
+\mathcal O(h^{1/8})\|u\|_{L^2}.
\end{equation}
\end{lemm}
\begin{proof}
Recall the density function $F$ from \eqref{e:fn-density}. By definition, the indicator function $\mathbf 1_{\mathcal Z}$ satisfies $0\leq \alpha\mathbf 1_{\mathcal Z}\leq F\leq 1$. Thus by Lemma~\ref{l:almost-positivity}
(where $A_F$ is defined by~\eqref{e:op-symbol-weight})
\begin{equation}
  \label{e:Yc-int-1}
\alpha\|A_{\mathcal Z}u\|_{L^2}\leq \|A_Fu\|_{L^2}
+\mathcal O(h^{1/8})\|u\|_{L^2}.
\end{equation}
Using the definition~\eqref{e:fn-density} together with~\eqref{e:A-total}
we rewrite $A_F$ as follows:
$$
A_F={1\over N_0}\sum_{j=0}^{N_0-1}\sum_{\mathbf w\in \mathcal W(N_0),w_j=1}A_{\mathbf w}
={1\over N_0}\sum_{j=0}^{N_0-1}(A_1+A_2)^{N_0-1-j}A_1(j) (A_1+A_2)^j.
$$
Recall that $\|A_1+A_2\|_{L^2\to L^2}\leq 1$, see the proof of Lemma~\ref{l:A-0}.
Then
$$
\|A_Fu\|_{L^2}\leq \max_{0\leq j<N_0}\|A_1(j)(A_1+A_2)^ju\|_{L^2}.
$$
Since $\|A_1(j)\|_{L^2\to L^2}=\|A_1\|_{L^2\to L^2}\leq C$ and
$(A_1+A_2)^j u-u$ can be estimated by Lemma~\ref{l:A-0}, we get
$$
\|A_Fu\|_{L^2}\leq 
\max_{0\leq j< N_0}\|A_1(j)u\|_{L^2}
+C\|(-h^2\Delta-I)u\|_{L^2}.
$$
Estimating $A_1(j)u$ by Lemma~\ref{l:propagated-control}, we get
\begin{equation}
  \label{e:Yc-int-2}
\|A_Fu\|_{L^2}\leq C\|\Op_h(a)u\|_{L^2}+{C\log(1/h)\over h}\|(-h^2\Delta-I)u\|_{L^2}+\mathcal O(h)\|u\|_{L^2}.
\end{equation}
Combining~\eqref{e:Yc-int-1} and~\eqref{e:Yc-int-2}, we obtain~\eqref{e:Yc-1}.
\end{proof}

\medskip

We now finish the proof of Proposition \ref{l:nontrapped-estimate}. Recalling~\eqref{e:XY-def}, we write
$$
\mathcal Y=\bigsqcup_{\ell=1}^{8}\mathcal Y_\ell,\quad
\mathcal Y_\ell := \{\mathbf w^{(1)}\dots \mathbf w^{(8)}\mid
\mathbf w^{(\ell)}\in \mathcal Z,\quad
\mathbf w^{(\ell+1)},\dots,\mathbf w^{(8)}\in\mathcal W(N_0)\setminus\mathcal Z\}.
$$
Then
$
A_{\mathcal Y}=\sum_{\ell=1}^8 A_{\mathcal Y_\ell}$.
Let $\mathcal{Q}:=\mathcal{W}(N_0)\setminus\mathcal{Z}$, then
using~\eqref{e:A-total} we have the following factorization:
$$
A_{\mathcal Y_\ell}=A_{\mathcal Q}(7N_0)\cdots A_{\mathcal Q}(\ell N_0)
A_{\mathcal Z}\big((\ell-1)N_0\big)(A_1+A_2)^{(\ell-1)N_0}.
$$
By Lemma~\ref{l:small-time-good} we have $\|A_{\mathcal Q}\|_{L^2\to L^2},\|A_{\mathcal Z}\|_{L^2\to L^2}\leq C$.
Estimating $(A_1+A_2)^{(\ell-1)N_0}u-u$ by Lemma~\ref{l:A-0}, we get
\begin{equation}
  \label{e:bunny-1}
\|A_{\mathcal Y}u\|_{L^2}\leq C\sum_{\ell=1}^8 \big\|A_{\mathcal Z}\big((\ell-1)N_0\big)u\big\|_{L^2}
+C\|(-h^2\Delta-I)u\|_{L^2}.
\end{equation}
We have by Lemma~\ref{l:propagated}
\begin{equation}
  \label{e:bunny-2}
\big\|A_{\mathcal Z}\big((\ell-1)N_0\big)u\big\|_{L^2}\leq
\|A_{\mathcal Z} u\|_{L^2}
+{C\log(1/h)\over h}\|(-h^2\Delta-I)u\|_{L^2}.
\end{equation}
Using Lemma~\ref{l:Y-control-1} to bound $\|A_{\mathcal Z}u\|_{L^2}$
and combining~\eqref{e:bunny-1} with~\eqref{e:bunny-2}, we obtain~\eqref{e:nontrapped-estimate},
finishing the proof.



\section{Fractal uncertainty principle}
  \label{s:fup}

In this section we prove Proposition~\ref{l:ultimate-fup}
using the fractal uncertainty principle established in~\cite{fullgap}.

\subsection{Fractal uncertainty principle for porous sets in $\mathbb R$}
  \label{s:fup-1d}
    
We start by adapting the result of~\cite{fullgap} to the setting of porous sets,
by embedding them into Ahlfors--David regular sets
of some dimension $\delta<1$. Here we define porous sets as follows:
\begin{defi}
  \label{d:porous}
Let $\nu\in(0,1)$ and $0<\alpha_0\leq\alpha_1$. We say that a subset $\Omega$ of $\mathbb{R}$ is \textbf{$\nu$-porous on scales $\alpha_0$ to $\alpha_1$} if for each interval $I$ of size $|I|\in[\alpha_0,\alpha_1]$, there exists a subinterval $J\subset I$ with $|J|=\nu|I|$ such that $J\cap \Omega=\emptyset$.
\end{defi}
As for Ahlfors--David regular sets, we recall
\begin{defi}\cite[Definition~1.1]{fullgap}
  \label{d:adreg}
Let $\delta\in[0,1]$, $C_R\geq 1$, and $0\leq \alpha_0\leq\alpha_1$.
We say that a closed nonempty subset $X$ of $\mathbb{R}$ is \textbf{$\delta$-regular with constant $C_R$ on scales $\alpha_0$ to $\alpha_1$} if there exists a Borel measure $\mu_X$ on $\mathbb{R}$ such that:
\begin{enumerate}
\item
$\mu_X$ is supported on $X$: $\mu_X(\mathbb{R}\setminus X)=0$;
\item
for any interval $I$ with $\alpha_0\leq|I|\leq\alpha_1$, we have $\mu_X(I)\leq C_R|I|^\delta$;
\item
if in addition $I$ is centered at a point in $X$, then $\mu_X(I)\geq C_R^{-1}|I|^\delta$.
\end{enumerate}
\end{defi}
We use the following version of fractal uncertainty principle for $\delta$-regular sets.
Henceforth for $X\subset\mathbb R$ and $s>0$,
$X(s)=X+[-s,s]$ denotes the $s$-neighborhood of $X$.
\begin{prop}\cite[Proposition 4.3]{fullgap}.
	\label{t:fullgap}
Let $B=B(h):L^2(\mathbb{R})\to L^2(\mathbb{R})$ be defined as
\begin{equation}
  \label{e:b-form}
Bf(x)=h^{-1/2}\int e^{i\Phi(x,y)/h}b(x,y)f(y)dy
\end{equation}
where $\Phi\in C^\infty(U;\mathbb{R})$, $b\in C_0^\infty(U)$,
$U\subset\mathbb{R}^2$ is open, and $\partial_{xy}^2\Phi\neq0$ on $U$.

Let $0\leq\delta<1$ and $C_R\geq 1$. Then there exist $\beta>0$, $\rho\in(0,1)$ depending only on $\delta, C_R$ and there exists $C>0$ depending only on $\delta, C_R, b,\Phi$ such that for all $h\in(0,1)$
and all $X,Y\subset\mathbb R$ which are $\delta$-regular
with constant $C_R$ on scales 0 to 1,
\begin{equation}
	\label{e:ful-regular}
\|\indic_{X(h^{\rho})}B(h)\indic_{Y(h^\rho)}\|_{L^2(\mathbb R)\to L^2(\mathbb R)}\leq Ch^\beta.
\end{equation}
\end{prop}
Although porous sets need not be regular, we can always embed a porous set $\Omega$ in a neighborhood of a $\delta$-regular set $X$ with $\delta<1$. The set $X$ is constructed
by a Cantor-like procedure with some large base $L$, where
at $k$-th step we remove intervals of size $L^{-k-1}$ which do not intersect
$\Omega$.
\begin{lemm}
  \label{l:porous-regular}
For each $\nu\in (0, 1)$ there exist $\delta=\delta(\nu)\in (0,1)$
and $C_R=C_R(\nu)\geq 1$ such that the following holds.
Let $\Omega$ be a $\nu$-porous set on scales $\alpha_0$ to $1$.
Then there exists a set $X$ which is $\delta$-regular with constant $C_R$ on scales 0 to 1 such that $\Omega\subset X(\alpha_0)$.
\end{lemm}
\begin{proof}
Put $L:=\lceil 2/\nu\rceil\in\mathbb N$. 
We use the tree of intervals
$$
I_{m,k}=[mL^{-k},(m+1)L^{-k}],\quad 
m,k\in\mathbb{Z}.
$$
Let $k_0\geq 0$ be the unique integer such that $L^{-1-k_0}<\alpha_0\leq L^{-k_0}$.

Take $m,k$ with $0\leq k\leq k_0$. We claim that there exists $n=n(m,k)$
such that
\begin{equation}
  \label{e:pore-1}
I_{n,k+1}\subset I_{m,k},\quad
I_{n,k+1}\cap\Omega=\emptyset.
\end{equation}
Indeed, since $\Omega$ is $\nu$-porous, there exists
a subinterval $J\subset I_{m,k}$ such that $|J|=\nu|I_{m,k}|\geq 2L^{-k-1}$
and $J\cap \Omega=\emptyset$. Then one can find $n$ such that
$I_{n,k+1}\subset J$, and this value of~$n$ satisfies~\eqref{e:pore-1}.
When $k>k_0$, we put $n(m,k):=Lm$, so that
the condition $I_{n(m,k),k+1}\subset I_{m,k}$ still holds.

We now define the set $X$ as follows:
$$
X:=\bigcap_{k=0}^\infty \overline{X_k},\quad
X_k:=\mathbb{R}\setminus\bigcup_{m\in\mathbb{Z}} I_{n(m,k),k+1}.
$$
Note that for each $k\geq 1$ there exists a set $\mathcal M(k)\subset \mathbb Z$ such that
$$
\bigcap_{\ell=0}^{k-1}\overline{X_\ell}=\bigcup_{m\in\mathcal M(k)}I_{m,k}.
$$
We set $\mathcal M(0):=\mathbb Z$. Then for all $k\geq 0$ and $m$ we have
\begin{equation}
  \label{e:pore-2}
\#\{m'\in\mathcal M(k+1)\mid I_{m',k+1}\subset I_{m,k}\}=\begin{cases}
L-1,& m\in \mathcal M(k);\\
0,&\text{otherwise.}
\end{cases}
\end{equation}
We claim that $\Omega\subset X(\alpha_0)$. Indeed,
by~\eqref{e:pore-1} we have
$\Omega\subset X_k$ when $0\leq k\leq k_0$.
Take $x\in \Omega$, then $x$ lies in
$\bigcap_{k=0}^{k_0}\overline{X_k}$,
which implies that
$x\in I_{m,k_0+1}$ for some $m\in\mathcal M(k_0+1)$.
Since $L\geq 2$, by induction using~\eqref{e:pore-2}
there exists a sequence $(m_k\in \mathcal M(k))_{k\geq k_0+1}$
with $m_{k_0+1}=m$ and $I_{m_{k+1},k+1}\subset I_{m_k,k}$.
The intersection $\bigcap_k I_{m_k,k}$ consists of a single
point $y\in X$.
Since $x,y\in I_{m,k_0+1}$
we have $|x-y|\leq L^{-k_0-1}$, thus $x\in X(L^{-k_0-1})\subset X(\alpha_0)$ as required.

It remains to prove that $X$ is $\delta$-regular with some constant $C_R$ on scales
0 to~1, where we put
$$
\delta:={\log(L-1)\over\log L}\in (0,1).
$$
Let $\mu_X$ be the natural Cantor-like measure supported on $X$.
More precisely, by~\eqref{e:pore-2} there exists a
unique Borel measure $\mu_X$ on $\mathbb R$ satisfying
for all $m$ and $k\geq 0$
$$
\mu_X(I_{m,k})=
\begin{cases}
(L-1)^{-k}=L^{-\delta k},&
m\in \mathcal M(k);\\
0,& \text{otherwise}.
\end{cases}
$$
Take an interval $I$ of size $|I|\leq 1$, and fix
the unique integer $k\geq 0$ such that $L^{-k-1}<|I|\leq L^{-k}$.
Then there exists $m$ such that $I\subset I_{m,k}\cup I_{m+1,k}$. It follows that
\begin{equation}
  \label{e:pore-3}
\mu_X(I)\leq \mu_X(I_{m,k})+\mu_X(I_{m+1,k})\leq 2L^{-\delta k}
\leq 2L\cdot |I|^\delta.
\end{equation}
Next, assume that $I$ is an interval of size $|I|\leq 1$ centered
at a point $x\in X$. Fix the unique integer $k\geq 0$ such that
$2L^{-k-1}\leq |I|<2L^{-k}$
and choose $m\in\mathcal M(k+1)$ such that
$x\in I_{m,k+1}$. Then
$I_{m,k+1}\subset I$ and thus
\begin{equation}
  \label{e:pore-4}
\mu_X(I)\geq \mu_X(I_{m,k+1})=L^{-\delta(k+1)}\geq {|I|^\delta\over 2L}.
\end{equation}
Recalling Definition~\ref{d:adreg}, we see
that~\eqref{e:pore-3} and~\eqref{e:pore-4}
imply that $X$ is $\delta$-regular with constant
$C_R:=2L$ on scales 0 to~1. This finishes the proof.
\end{proof}
Combining Proposition~\ref{t:fullgap} and Lemma~\ref{l:porous-regular},
we obtain the following fractal uncertainty principle for $\nu$-porous sets:
\begin{prop}
	\label{l:fup-1d-porous}
Let $K>0$ and $\nu\in(0,1)$ be fixed
and $B(h):L^2(\mathbb R)\to L^2(\mathbb R)$ be as in Proposition~\ref{t:fullgap}.
Then there exist $\beta>0$, $\rho\in (0,1)$ depending only on $\nu$
and there exists $C$ depending only on $\nu,K,b,\Phi$ such that
for all $h\in (0,1)$ and all $\Omega_\pm\subset\mathbb R$
which are $\nu$-porous on scales $Kh^\rho$ to 1,
\begin{equation}
	\label{e:fup-porous}
\|\indic_{\Omega_-(Kh^\rho)}B(h)\indic_{\Omega_+(Kh^\rho)}\|_{L^2\to L^2}\leq Ch^\beta.
\end{equation}
\end{prop}
\begin{proof}
By Lemma~\ref{l:porous-regular}, there exist
$X,Y\subset\mathbb R$ which are $\delta$-regular with constant
$C_R$ on scales 0 to~1 for some $\delta=\delta(\nu)\in (0,1)$, $C_R=C_R(\nu)$
such that
$$
\Omega_-\subset X(Kh^\rho),\quad
\Omega_+\subset Y(Kh^\rho).
$$
Then
$$
\|\indic_{\Omega_-(Kh^\rho)}B(h)\indic_{\Omega_+(Kh^\rho)}\|_{L^2\to L^2}
\leq \|\indic_{X(2Kh^\rho)}B(h)\indic_{Y(2Kh^\rho)}\|_{L^2\to L^2}.
$$
It remains to apply Proposition~\ref{t:fullgap}
where we increase $\rho$ slightly to absorb the constant~$2K$.
\end{proof}


\subsection{Fractal uncertainty principle for porous sets in $T^* M$}
  \label{s:fup-porous}

We next use Proposition~\ref{l:fup-1d-porous} to prove a fractal uncertainty
principle for subsets of $T^*M\setminus 0$, where $M$ is a compact orientable hyperbolic surface.

Let $H_p,U_+,U_-,D$ be the frame on $T^*M\setminus 0$ defined in~\eqref{e:canonical-fields}.
For $v=(v_1,v_2,v_3)\in \mathbb R^3$, define the vector fields
$$
\mathcal V_\pm v=v_1H_p+v_2D+v_3 U_\pm.
$$
For $(x,\xi)\in T^*M\setminus 0$ and $\nu_0,\nu_1>0$, we define
the \emph{stable $(\nu_0,\nu_1)$ slice centered at $(x,\xi)$}
as follows:
$$
\Sigma^+_{\nu_0,\nu_1}(x,\xi):=
\{\exp(\mathcal V_- v)\exp(sU_+)(x,\xi)\colon |s|\leq \nu_0,\
|v|\leq \nu_1\}.
$$
Similarly define the \emph{unstable $(\nu_0,\nu_1)$ slice
centered at $(x,\xi)$:}
$$
\Sigma^-_{\nu_0,\nu_1}(x,\xi):=
\{\exp(\mathcal V_+ v)\exp(sU_-)(x,\xi)\colon |s|\leq \nu_0,\
|v|\leq \nu_1\}.
$$
\begin{figure}
\includegraphics{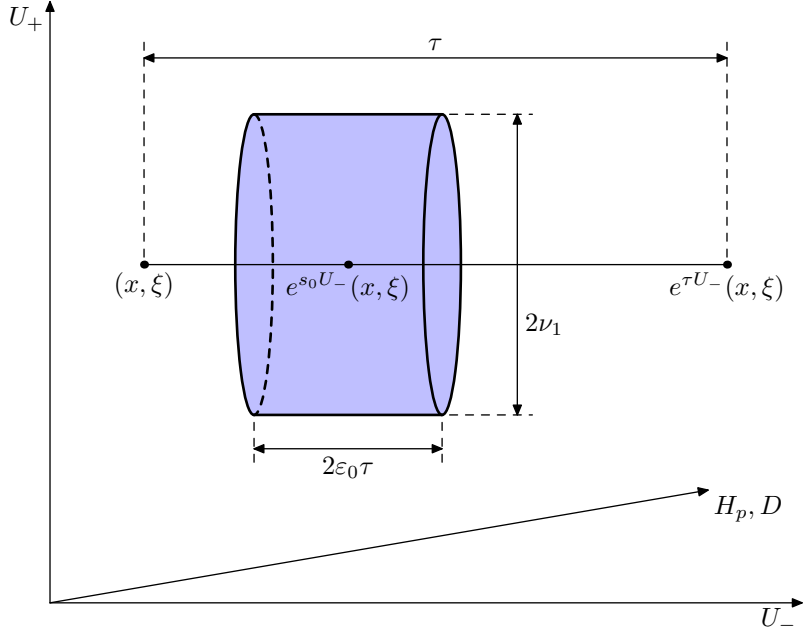}
\caption{An illustration of Definition~\ref{d:porous-TM} of an $(\varepsilon_0,\nu_1)$-porous
set along $U_-$. The blue cylinder is the unstable slice $\Sigma^-_{\varepsilon_0\tau,\nu_1}(e^{s_0U_-}(x,\xi))$.
(We ignore here the fact that $H_p,U_\pm,D$ do not commute and thus do not give rise to a coordinate system.)}
\label{f:porous}
\end{figure}
%
\begin{defi}
  \label{d:porous-TM}
Let
$$
Z\subset \{1/4\leq |\xi|_g\leq 4\}\subset T^*M\setminus 0
$$
be a closed set and fix
$$
\varepsilon_0,\nu_1,\tau_0\in (0,1].
$$
We say that
$Z$ is \textbf{$(\varepsilon_0,\nu_1)$-porous along
$U_\pm$ up to scale $\tau_0$}, if for each $(x,\xi)\in T^*M\setminus 0$
and each $\tau\in[\tau_0,1]$, there exists
$s_0\in [0,\tau]$ such that (see Figure~\ref{f:porous})
$$
\Sigma^\pm_{\varepsilon_0 \tau,\nu_1}(e^{s_0 U_\pm}(x,\xi))\cap Z=\emptyset.
$$
\end{defi}
Our fractal uncertainty principle for subsets of $T^*M\setminus 0$ is formulated
in terms of the $\Psi^{\comp}_{h,L,\rho}(T^*M\setminus 0)$ calculus introduced in~\S\ref{s:fancy-calculus}:
\begin{prop}
	\label{l:fup-porous}
Fix $\varepsilon_0,\nu_1\in (0,1]$. Then there exist $\beta>0$ and $\rho\in (0,1)$
depending only on $M,\varepsilon_0,\nu_1$ such that the following holds.
Suppose that 
$$
a_+\in S^{\comp}_{L_u,\rho}(T^*M\setminus 0),\quad
a_-\in S^{\comp}_{L_s,\rho}(T^*M\setminus 0),
$$
and $\supp a_\pm$ is $(\varepsilon_0,\nu_1)$-porous along $U_\pm$
up to scale $K_1h^\rho$ for some constant $K_1$.
Then for all $Q\in\Psi^0_h(M)$
\begin{equation}
  \label{e:fup-2}
\|\Op_h^{L_s}(a_-)Q\Op_h^{L_u}(a_+)\|_{L^2\to L^2}\leq Ch^\beta
\end{equation}
where $C$ depends only on $M,\varepsilon_0,\nu_1,K_1,Q$, and
some $S^{\comp}_{\bullet,\rho}$ seminorms
of $a_\pm$.
\end{prop}
In the rest of this subsection, we prove Proposition~\ref{l:fup-porous}.
We begin by straightening out weak
stable/unstable Lagrangian foliations similarly to~\cite[\S4.4]{hgap}.
Denote by~$\mathbb H^2$ the hyperbolic
plane; it is the universal cover of $M$. Let
$$
\varkappa^\pm:(x,\xi)\in T^*\mathbb H^2\setminus 0\mapsto (w,y,\theta,\eta)\in T^*(\mathbb R^+_w\times\mathbb S^1_y)
$$
be the exact symplectomorphisms constructed in~\cite[Lemma~4.7]{hgap} mapping $L_s,L_u$ to the vertical foliation $L_0$ on $T^\ast(\mathbb{R}^+\times\mathbb{S}^1)$:
$$
(\varkappa^+)_*L_u=(\varkappa^-)_*L_s=L_0=\ker(dw)\cap\ker(dy).
$$ 
More precisely, in the Poincar\'e disk model of $\mathbb H^2$, we have $w=p(x,\xi)=|\xi|_g$,
$$
y=B_\mp(x,\xi)
$$
is the limit of the projection to $\mathbb H^2$ of the geodesic $e^{tH_p}(x,\xi)$ as $t\to\mp\infty$ on the boundary $\mathbb S^1=\partial\mathbb H^2$,
$$
\theta=\pm\log\mathcal{P}(x,B_\mp(x,\xi)),
$$
where
$$
\mathcal{P}(x,y)=\frac{1-|x|^2}{|x-y|^2},\quad x\in\mathbb H^2,\ y\in\mathbb S^1
$$
is the Poisson kernel, and
$$
\eta=\pm G_\mp(x,\xi)=\pm p(x,\xi)\mathcal{G}(B_\mp(x,\xi),B_\pm(x,\xi))\in T_{B_\mp(x,\xi)}^*\mathbb S^1
$$
where (see~\cite[(1.19)]{hgap})
$$
\mathcal{G}(y,y')=\frac{y'-(y\cdot y')y}{1-y\cdot y'}\in
T^*_y\mathbb S^1\simeq T_y\mathbb S^1\subset \mathbb{R}^2,\quad
y,y'\in\mathbb S^1,\
y\neq y'
$$
is half the stereographic projection of $y'$ with base $y$.
See Figure~\ref{f:coord}.
\begin{figure}
\includegraphics{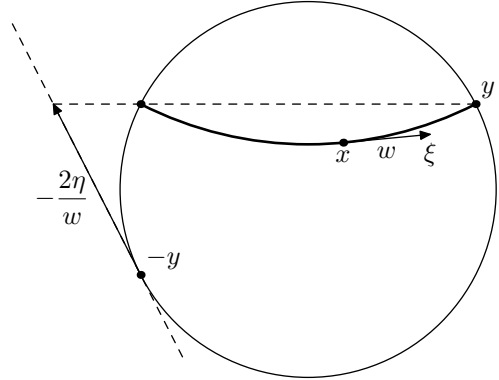}
\caption{The coordinates $(w,y,\theta,\eta)=\varkappa_-(x,\xi)$ in the
Poincar\'e disk model of $\mathbb H^2$. Here $w$ is the length of $\xi$,
$y$ is the limit of the geodesic starting from $(x,\xi)$ at $t\to\infty$,
$\theta$ is determined from the Poisson kernel $\mathcal P(x,y)$,
and $\eta$ is determined from the stereographic projection pictured.
By~\eqref{e:foliated-correctly} the value of~$y$ does not change
if we deform $(x,\xi)$ along the stable, flow, or dilation direction.}
\label{f:coord}
\end{figure}

It follows from the definition of~$B_\pm(x,\xi)$ that
\begin{equation}
  \label{e:foliated-correctly}
(\mathcal V_\pm v)B_\pm=0\quad\text{for all }v\in\mathbb R^3.
\end{equation}
By a microlocal partition of unity
and since $\supp a_\pm\subset \{1/4\leq |\xi|_g\leq 4\}$, we can assume that $\WFh(Q)\subset V$ where $V$ is a sufficiently small neighborhood of any given point $(x_0,\xi_0)\in T^*M\setminus 0$. We assume that
\begin{equation}
  \label{e:diam-V}
\diam(V)\leq \nu_1/C_0
\end{equation}
where $C_0$ is a large constant depending only on $M$ to be chosen
in Lemma~\ref{l:porosity-conversion} below.
We lift $V\subset T^*M\setminus 0$ to a subset of $T^*\mathbb{H}^2\setminus0$ and use $\varkappa^\pm$ to define the symplectomorphisms onto their images
$$
\varkappa^\pm_0:V\to T^*(\mathbb R^+\times\mathbb S^1).
$$
Note that we can make $\varkappa^\pm_0(V)$ contained in a compact subset of
$T^*(\mathbb R^+\times\mathbb S^1)$ which only depends on $M$.

We next quantize $\varkappa_0^\pm$ by Fourier integral operators
which conjugate $\Op_h^{L_s}(a_-)$ and $\Op_h^{L_u}(a_+)$ to operators
on $\mathbb R^+\times\mathbb S^1$.
Following~\cite[\S4.4, Proof of Theorem~3]{hgap}, we consider operators
$$
\mathcal B_\pm\in I_h^{\comp}(\varkappa_0^\pm),\quad
\mathcal B'_\pm\in I_h^{\comp}((\varkappa_0^\pm)^{-1})
$$
quantizing $\varkappa_0^\pm$ near $\varkappa_0^\pm(\WFh(Q))\times \WFh(Q)$ in the sense of~\eqref{e:fio-quantize}.
Consider the following operators on $L^2(\mathbb R^+\times\mathbb S^1)$:
$$
A_-:=\mathcal B_-\Op_h^{L_s}(a_-)\mathcal B_-',\quad
A_+:=\mathcal B_+Q\Op_h^{L_u}(a_+)\mathcal B_+',\quad
B=\mathcal B_-\mathcal B_+'.
$$
Then similarly to~\cite[(4.58)]{hgap}
$$
\Op_h^{L_s}(a_-)Q\Op_h^{L_u}(a_+)=\mathcal B'_-A_-BA_+\mathcal B_++\mathcal O(h^\infty)_{L^2\to L^2}.
$$
Moreover, by~\eqref{e:genius-3} there exist $\tilde a_\pm\in S^{\comp}_{L_0,\rho}(T^*(\mathbb R^+\times\mathbb S^1))$ such that
\begin{equation}
  \label{e:A-pm-def}
A_\pm=\Op_h^{L_0}(\tilde a_\pm)+\mathcal O(h^\infty)_{L^2\to L^2},\quad
\supp\tilde a_\pm\subset \varkappa_0^\pm(V\cap \supp a_\pm).
\end{equation}
Therefore in order to establish~\eqref{e:fup-2} it suffices to prove that
\begin{equation}
  \label{e:fup-3}
\|\Op_h^{L_0}(\tilde a_-)B\Op_h^{L_0}(\tilde a_+)\|_{L^2(\mathbb R^+\times\mathbb S^1)\to L^2(\mathbb R^+\times\mathbb S^1)}\leq Ch^\beta.
\end{equation}
Using the porosity of $\supp a_\pm$ along $U_\pm$, we get the following
one-dimensional porosity statement for projections of $\supp\tilde a_\pm$:
\begin{lemm}
  \label{l:porosity-conversion}
There exists a constant $C_0>0$ depending only on $M$ such that the following holds.
Define the projections of $\supp \tilde a_\pm$ onto the $y$ variable
$$
\Omega_\pm:=\big\{y\in\mathbb S^1\mid \exists w,\theta,\eta\colon (w,y,\theta,\eta)\in \supp\tilde a_\pm\big\}\ \subset\ \mathbb S^1.
$$
Then $\Omega_\pm$ (more precisely, their lifts to $\mathbb R$) are $\nu$-porous on scales $\alpha_0$ to~1 in the sense of Definition~\ref{d:porous}
where $\nu:=\varepsilon_0\nu_1/C_0$
and $\alpha_0:=C_0\nu_1^{-1}K_1 h^\rho$.
\end{lemm}
\begin{proof}
We show the porosity of $\Omega_+$, with the case of $\Omega_-$ handled
similarly. 
Denote by $C_1>0$ a large constant depending only on $M$
and put $C_0:=C_1^4$.

Denote $W:=\varkappa_0^+(V)$.
Let $V'$ be the $\nu_1/C_1^2$-neighborhood of $V$ and
$V''$ be the $\nu_1/C_1^2$-neighborhood of $V'$.
Lifting $V''$ to $T^*\mathbb H^2\setminus 0$ and 
using $\varkappa^+$, we extend $\varkappa^+_0$ to a symplectomorphism
$$
\varkappa^+_0:V'\to W',\
V''\to W''
$$
for some open sets $W',W''\subset T^*(\mathbb R^+\times\mathbb S^1)$.
Note that by~\eqref{e:diam-V}
\begin{equation}
  \label{e:diam-W}
\diam(W'')\leq {C_1\over 10}\diam(V'')\leq \nu_1/C_1.
\end{equation}
Moreover, the $\nu_1/C_1^3$-neighborhoods of $W,W'$ are contained
in $W',W''$ respectively.

Let $I\subset\mathbb S^1$ be an interval with $\alpha_0\leq |I|\leq 1$
centered at some $y_0\in\mathbb S^1$. Assume first that
the $y$-projection of $W'$ does not contain $y_0$. Then,
since $\supp\tilde a_+\subset W$ by~\eqref{e:A-pm-def}, we see that 
$y_0$ lies distance at least $\nu_1/C_0$ away from $\Omega_+$.
Thus the interval of size $\nu|I|$ centered at $y_0$
does not intersect $\Omega_+$ and verifies the porosity condition
in Definition~\ref{d:porous}.

We henceforth assume that the $y$-projection of $W'$ does contain $y_0$.
Choose $w_0,\theta_0,\eta_0$ such that
$(w_0,y_0,\theta_0,\eta_0)\in W'$.
Let $(x_0,\xi_0):=(\varkappa_0^+)^{-1}(w_0,y_0,\theta_0,\eta_0)\in V'$. Put
$$
\tau:=C_1^{-3}\nu_1|I|,\quad
K_1h^\rho\leq \tau\leq \nu_1/C_1^3\leq 1.
$$
Since $\supp a_+$ is $(\varepsilon_0,\nu_1)$-porous along $U_+$ up to scale $K_1h^\rho$,
there exists $s_0\in [0,\tau]$ such that
\begin{equation}
  \label{e:gold}
\Sigma^+_{\varepsilon_0\tau,\nu_1}(x_1,\xi_1)\cap \supp a_+=\emptyset\quad\text{where }
(x_1,\xi_1):=e^{s_0U_+}(x_0,\xi_0)\in V''.
\end{equation}
Since $C_1$ is large and $H_p,U_+,U_-,D$ form a frame,
we have a diffeomorphism
$$
\Theta:\widetilde U\to W'',\quad
(s,v)\mapsto \varkappa_0^+\big(\exp(\mathcal V_-v)\exp(sU_+)(x_1,\xi_1)\big)
$$
where $\widetilde U$ is some neighborhood of $(0,0)$ in $\mathbb R\times\mathbb R^3$.
By~\eqref{e:foliated-correctly} we see
that for $(w,y,\theta,\eta)=\Theta(s,v)$, the value of $y$ does not change
if we change $v$. Therefore the $y$-component of $\Theta(s,v)$
is equal to $\Theta_1(s)$ for some smooth diffeomorphism $\Theta_1$
defined on a subset of $\mathbb R$.

Applying $\varkappa_0^+$ to~\eqref{e:gold} and using~\eqref{e:A-pm-def}, we get
\begin{equation}
  \label{e:silver}
\{\Theta(s,v)\colon (s,v)\in\widetilde U,\ |s|\leq \varepsilon_0\tau,\
|v|\leq\nu_1\}\cap \supp \tilde a_+=\emptyset.
\end{equation}
However, by~\eqref{e:diam-W} we have
$$
\diam(\widetilde U)\leq \sqrt{C_1}\diam(W'')\leq {\nu_1\over 10}
$$
and thus the condition $|v|\leq\nu_1$ in~\eqref{e:silver} is not needed.
Therefore
\begin{equation}
  \label{e:silver2}
\Theta_1^{-1}(\Omega_+)\cap [-\varepsilon_0\tau,\varepsilon_0\tau]=\emptyset.
\end{equation}
Denote
$$
(w_1,y_1,\theta_1,\eta_1):=\Theta(0,0)=\varkappa_0^+(x_1,\xi_1)\in W''.
$$
and consider the interval
$$
J:=\big[y_1,y_1+\nu|I|\big],\quad
|J|=\nu|I|.
$$
We have $|y_0-y_1|\leq C_1s_0\leq C_1^{-2}\nu_1|I|$. Therefore $J\subset I$.
Moreover, since $\Theta_1(0)=y_1$
and $\diam(\Theta_1^{-1}(J))\leq C_1\nu|I|\leq \varepsilon_0\tau$, \eqref{e:silver2} implies that
$J\cap\Omega_+=\emptyset$. This gives the required porosity condition
on $\Omega_+$.
\end{proof}

\medskip

We are ready to finish the proof of Proposition~\ref{l:fup-porous}.
The operator $B=\mathcal B_-\mathcal B_+'$ lies
in $I^{\comp}_h\big(\varkappa^-\circ(\varkappa^+)^{-1}\big)$.
By~\cite[Lemma~4.9]{hgap} we can write
$$
B=A\widetilde{\mathcal B}_\chi+\mathcal O(h^\infty)_{L^2\to L^2}\quad\text{for some }A\in\Psi^{\comp}_h(\mathbb R^+\times\mathbb S^1)
$$
where $\chi\in C_0^\infty(\mathbb S^1_y\times\mathbb S^1_{y'})$, $\supp\chi\subset \{y\neq y'\}$,
and $\widetilde{\mathcal B}_\chi:L^2(\mathbb R^+\times\mathbb S^1)\to L^2(\mathbb R^+\times\mathbb S^1)$ is given by
$\widetilde{\mathcal B}_\chi v(w,y)=\mathcal B_{\chi,w}(v(w,\bullet))(y)$ where
$$
\mathcal B_{\chi,w} v(y)=
(2\pi h)^{-1/2}\int_{\mathbb S^1}\Big|{y-y'\over 2}\Big|^{2iw/h}\chi(y,y')v(y')\,dy',\quad
w>0.
$$
Here $|y-y'|$ denotes the Euclidean distance between $y,y'\in\mathbb S^1\subset\mathbb R^2$.

Since $\supp a_\pm\subset \{1/4\leq |\xi|_g\leq 4\}$, we have
$\supp \tilde a_\pm\subset \{1/4\leq w\leq 4\}$.
We can write
\begin{equation}
  \label{e:wallaby}
\Op_h^{L_0}(\tilde a_-)B\Op_h^{L_0}(\tilde a_+)
=\Op_h^{L_0}(a_-')\widetilde{\mathcal B}_{\chi}\Op_h^{L_0}(a_+')+\mathcal O(h^\infty)_{L^2\to L^2}
\end{equation}
where $a_\pm'\in S^{\comp}_{L_0,\rho}(T^*(\mathbb R^+\times\mathbb S^1))$
satisfy
\begin{equation}
  \label{e:a-pm-prime}
\supp a_\pm'\subset \{1/4\leq w\leq 4,\ y\in \Omega_\pm\}.
\end{equation}
In fact, $a_-'=\tilde a_-\#\sigma_h(A)$ and $a_+'=\tilde a_+$.
By~\cite[Lemma~3.3]{hgap} there exist symbols $\chi_\pm(y;h)$ such that
$$
|\partial^k_y \chi_\pm|\leq C_k h^{-\rho k},\quad
\supp(1-\chi_\pm)\cap \Omega_\pm=\emptyset,\quad
\supp\chi_\pm\subset \Omega_\pm(h^{\rho}).
$$
Take also $\chi_w(w)\in C_0^\infty((1/8,8))$ such that $\chi_w=1$ near $[1/4,4]$. Then it
follows from~\eqref{e:wallaby} and~\eqref{e:a-pm-prime} that
$$
\Op_h^{L_0}(\tilde a_-)B\Op_h^{L_0}(\tilde a_+)
=\Op_h^{L_0}(a_-')\chi_w\chi_-\widetilde{\mathcal B}_{\chi}\chi_+\Op_h^{L_0}(a_+')+\mathcal O(h^\infty)_{L^2\to L^2}.
$$
Therefore~\eqref{e:fup-3} follows from the estimate
$$
\|\chi_w\chi_-\widetilde{\mathcal B}_\chi\chi_+\|_{L^2(\mathbb R^+\times\mathbb S^1)\to L^2(\mathbb R^+\times\mathbb S^1)}\leq Ch^\beta
$$
which in turn follows from
\begin{equation}
  \label{e:fup-4}
\sup_{w\in [1/8,8]}\|\indic_{\Omega_-(h^\rho)}\mathcal B_{\chi,w}\indic_{\Omega_+(h^\rho)}\|_{L^2(\mathbb S^1)\to L^2(\mathbb S^1)}
\leq Ch^\beta.
\end{equation}
The operator $\mathcal B_{\chi,w}$ has the form~\eqref{e:b-form}
with $\Phi(y,y')=2w\log|y-y'|-w\log 4$, $y,y'\in\mathbb S^1$, $y\neq y'$,
where we pass from operators on $\mathbb S^1$ to operators on $\mathbb R$
by taking a partition of unity for $\chi$.
The mixed derivative $\partial^2_{yy'}\Phi$ does not vanish
as verified for instance in~\cite[\S4.3]{fullgap}.
Therefore~\eqref{e:fup-4} follows from the one-dimensional
fractal uncertainty principle, Proposition~\ref{l:fup-1d-porous},
where the porosity condition for $\Omega_\pm$
has been verified in Lemma~\ref{l:porosity-conversion}.

\subsection{Proof of Proposition~\ref{l:ultimate-fup}}
\label{s:pf-fup}

We now prove Proposition~\ref{l:ultimate-fup}. Take an arbitrary word $\mathbf{w}\in\mathcal{W}(2N_1)$
and write it as a concatenation of two words in $\mathcal{W}(N_1)$:
$$
\mathbf w=\mathbf w_+\mathbf w_-,\quad
\mathbf w_\pm\in \mathcal W(N_1).
$$
Define the operators
$$
\mathcal A_+:=A_{\mathbf w_+}(-N_1),\quad
\mathcal A_-:=A_{\mathbf w_-}.
$$
Then
\begin{equation}
  \label{e:endgame-1}
A_{\mathbf w}=
U(-N_1)\mathcal A_-\mathcal A_+U(N_1).
\end{equation}
We relabel the letters in the words $\mathbf{w}_\pm$ as follows:
$$
\mathbf w_+=w^+_{N_1}\dots w^+_1,\quad
\mathbf w_-=w^-_0\dots w^-_{N_1-1}
$$
and define the symbols $a_\pm$ by
$$
a_+=\prod_{j=1}^{N_1} (a_{w_j^+}\circ \varphi_{-j}),\quad
a_-=\prod_{j=0}^{N_1-1} (a_{w_j^-}\circ \varphi_j).
$$
Recall from~\eqref{e:op-word} that
$$
\begin{aligned}
\mathcal A_-&=A_{w^-_{N_1-1}}(N_1-1)A_{w^-_{N_1-2}}(N_1-2)\cdots A_{w_1^-}(1)A_{w_0^-}(0),\\
\mathcal A_+&=A_{w^+_1}(-1)A_{w^+_2}(-2)\cdots A_{w^+_{N_1-1}}(1-N_1)A_{w^+_{N_1}}(-N_1).
\end{aligned}
$$
\begin{lemm}
  \label{l:good-symbols}  
The symbols $a_\pm$ and the operators $\mathcal A_\pm$ satisfy 
$$
\begin{aligned}
a_+\in S^{\comp}_{L_u,\rho}(T^*M\setminus 0),&\quad
a_-\in S^{\comp}_{L_s,\rho}(T^*M\setminus 0);\\
\mathcal A_+=\Op_h^{L_u}(a_+)
+\mathcal O(h^{1-\rho-})_{L^2\to L^2},&\quad
\mathcal A_-=\Op_h^{L_s}(a_-)
+\mathcal O(h^{1-\rho-})_{L^2\to L^2}.
\end{aligned}
$$
\end{lemm}
\begin{proof}
The statement for $a_-$ and $\mathcal A_-$ follows directly from Lemma \ref{l:long-word-egorov}. The statement for $a_+$ and $\mathcal A_+$ can be obtained similarly by reversing the flow $\varphi_t$ which exchanges the stable and unstable foliations.
\end{proof}
By Lemma~\ref{l:good-symbols} and~\eqref{e:endgame-1}, to show Proposition~\ref{l:ultimate-fup} it suffices to
prove the estimate
$$
\|\Op_h^{L_s}(a_-)\Op_h^{L_u}(a_+)\|_{L^2\to L^2}\leq Ch^\beta.
$$
The latter follows from the version of the fractal uncertainty principle in
Proposition~\ref{l:fup-porous} (with $Q=I$) where the porosity condition is established by the following
\begin{lemm}
  \label{l:porosity-verified}
There exist $\varepsilon_0,\nu_1,K_1>0$ depending only
on $M,\mathcal U_1,\mathcal U_2$ such that
the sets $\supp a_\pm$ are $(\varepsilon_0,\nu_1)$-porous 
up to scale $K_1h^\rho$ along $U_\pm$ in the sense of Definition~\ref{d:porous-TM}.
\end{lemm}
\begin{proof}
We show the porosity of $\supp a_-$. The porosity of $\supp a_+$
can be proved in the same way, by reversing the direction of the flow $\varphi_t$.

Recall from~\eqref{e:A-12-2} that
$\supp a_1\cap \mathcal U_1=\supp a_2\cap\mathcal U_2=\emptyset$
where $\mathcal U_1,\mathcal U_2$ are nonempty open conic
subsets of $T^*M\setminus 0$. Fix nonempty open conic subsets
$\mathcal U'_1,\mathcal U'_2\subset T^*M\setminus 0$ such that
$\mathcal U'_w\cap S^*M\Subset \mathcal U_w$, $w=1,2$.

By Proposition~\ref{l:horocycle-unique} and since the vector field $U_-$ is homogeneous,
there exists $T>1$ depending only on
$M,\mathcal  U'_1,\mathcal U'_2$ such that for each $(x,\xi)\in T^*M\setminus 0$,
there exist $s_w=s_w(x,\xi)\in [0,T]$, $w=1,2$,
such that
$$
\exp(s_wU_-)(x,\xi)\in \mathcal U'_w.
$$
We put $K_1:=3T$. Take arbitrary $(x,\xi)\in T^*M\setminus 0$
and $\tau$ such that $K_1 h^\rho\leq \tau\leq 1$.
Let $j$ be the unique integer such that $e^{j-1}\tau < T \leq e^j\tau$, then $1\leq j\leq N_1-1$.
Denote $w:=w_j^-\in \{1,2\}$, so that
\begin{equation}
  \label{e:endgame-2}
\supp a_-\cap \varphi_{-j}(\mathcal U_w)=\emptyset.
\end{equation}
Since $e^j\tau\geq T$, we see that there exists $s_0:=e^{-j}s_w(\varphi_j(x,\xi))\in [0,\tau]$ such that
\begin{equation}
  \label{e:endgame-3}
q:=\varphi_j\big(\exp(s_0U_-)(x,\xi)\big)=\exp(e^j s_0U_-)\big(\varphi_j(x,\xi)\big)\in \mathcal U'_w.
\end{equation}
Here we used the commutation relations~\eqref{e:comm-rel}.
For $v\in\mathbb R^3$ and $s\in\mathbb R$ we have
\begin{equation}
  \label{e:endgame-4}
\varphi_j\big(\exp(\mathcal V_+v)\exp((s+s_0)U_-)(x,\xi)\big)=
\exp(\mathcal V_+v')\exp(e^jsU_-)(q)
\end{equation}
where $v'=(v_1,v_2,e^{-j}v_3)$, in particular $|v'|\leq |v|$. Now, choose $\nu_1>0$ such that for $w=1,2$
$$
\max(|v|,|s|)\leq\nu_1\quad\Longrightarrow\quad
e^{\mathcal V_+ v}e^{sU_-}(\mathcal U'_w)\subset \mathcal U_w
$$
and put $\varepsilon_0:={\nu_1/(3T)}$. By~\eqref{e:endgame-3} and~\eqref{e:endgame-4}
we have $\Sigma^-_{\varepsilon_0\tau,\nu_1}(e^{s_0U_-}(x,\xi))\subset \varphi_{-j}(\mathcal U_w)$.
By~\eqref{e:endgame-2}
we then have $\Sigma^-_{\varepsilon_0\tau,\nu_1}(e^{s_0U_-}(x,\xi))\cap \supp a_-=\emptyset$. This finishes the proof of the porosity of $\supp a_-$.
\end{proof}

\appendix
\section*{Appendix: Calculus associated to a Lagrangian foliation}
\renewcommand{\theequation}{A.\arabic{equation}}
\refstepcounter{section}
\renewcommand{\thesection}{A}
\setcounter{equation}{0}

In this appendix, we establish properties of the $\Psi^{\comp}_{h,L,\rho}$ pseudodifferential calculus introduced in~\S\ref{s:fancy-calculus}.
We follow~\cite[\S3]{hgap}, indicating the changes necessary. We present the calculus
in the general setting of a Lagrangian foliation on an arbitrary manifold.

\subsection{Symbols}
  \label{s:appendix-symbols}

We assume that $M$ is a manifold, $U\subset T^*M$ is an open set,
and $L$ is a Lagrangian foliation, that is for each $(x,\xi)\in U$,
$L_{(x,\xi)}\subset T_{(x,\xi)}(T^*M)$ is a Lagrangian subspace
depending smoothly on $(x,\xi)$ and the family $(L_{(x,\xi)})_{(x,\xi)\in U}$ is integrable.
See~\cite[Definition~3.1]{hgap}.

To keep track of powers of $h$ in the remainders, we introduce a slightly
more general class of symbols than the one used in~\S\ref{s:fancy-calculus}.
Fix two parameters
$$
0\leq \rho<1,\quad
0\leq \rho'\leq {\rho\over 2},\quad
\rho+\rho' < 1.
$$
We say that an $h$-dependent symbol $a$
lies in the class $S^{\comp}_{L,\rho,\rho'}(U)$ if
\begin{enumerate}
\item $a(x,\xi;h)$ is smooth in $(x,\xi)\in U$,
defined for $0<h\leq 1$,
and supported in an $h$-independent compact
subset of $U$;
\item $a$ satisfies the derivative bounds
\begin{equation}
  \label{e:symbol-derby-general}
\sup_{x,\xi}|Y_1\ldots Y_mZ_1\ldots Z_k a(x,\xi;h)|\leq Ch^{-\rho k-\rho'm},\quad
0<h\leq 1
\end{equation}
for all vector fields
$Y_1,\dots,Y_m,Z_1,\dots,Z_k$ on $U$ such that
$Y_1,\dots,Y_m$ are tangent to $L$.
Here the constant $C$ depends on $Y_1,\dots,Y_m,Z_1,\dots,Z_k$
but does not depend on $h$.
\end{enumerate}
For $\rho'=0$ we obtain the class used in~\cite[\S3]{hgap}. Moreover, the
class $S^{\comp}_{L,\rho}(T^*M\setminus 0)$ introduced in~\S\ref{s:fancy-calculus} is given by
\begin{equation}
  \label{e:funny-epsilons}
S^{\comp}_{L,\rho}(T^*M\setminus 0)=\bigcap_{\varepsilon>0} S^{\comp}_{L,\rho+\varepsilon,\varepsilon}(T^*M\setminus 0).
\end{equation}
In the arguments below
(for instance, in~\eqref{e:beepy-2.5}, \eqref{e:beepy-3.5},
and~\eqref{e:gijoe-3})
we implicitly use the following version of Borel's Theorem
(see~\cite[Theorem~4.15]{e-z} for the standard version whose
proof applies here).
Let 
$a_j\in S^{\comp}_{L,\rho,\rho'}(U)$ be
a sequence of symbols with supports contained
in a compact subset of $U$ independent of $h,j$.
Take an increasing sequence 
of real numbers
$m_j\geq 0$, $m_j\to \infty$. Then there exists a symbol $a\in S^{\comp}_{L,\rho,\rho'}(U)$ which is
an asymptotic sum of $h^{m_j}a_j$ in the following sense:
$$
a-\sum_{j=0}^{J-1} h^{m_j}a_j\in h^{m_J}S^{\comp}_{L,\rho,\rho'}(U)\quad\text{for all }J
$$
and moreover $\supp a\subset \bigcup_j\supp a_j$.
Here $\supp a$ denotes the support of $a$ in the $(x,\xi)$ variables,
which is an $h$-dependent family of compact subsets of $U$.

We have the following bound for the product of many
symbols in~$S^{\comp}_{L,\rho,\rho'}(U)$:
\begin{lemm}
  \label{l:long-product-symbols}
Let $C$ be an arbitrary fixed constant and assume that $a_1,\dots,a_N\in S^{\comp}_{L,\rho,\rho'}(U)$, $1\leq N\leq C\log (1/h)$
are such that $\sup |a_j|\leq 1$ and
each $S^{\comp}_{L,\rho,\rho'}(U)$ seminorm of $a_j$ is bounded uniformly in $j$.
Then for all small $\varepsilon>0$
the product $a_1\cdots a_N$ lies in $S^{\comp}_{L,\rho+\varepsilon,\rho'+\varepsilon}(U)$.
\end{lemm}
\begin{proof}
We see immediately that $\sup |a_1\cdots a_N|\leq 1$ and $\supp (a_1\cdots a_N)\subset\supp a_1$
lies in an $h$-independent compact subset of $U$.
It remains to verify that for all vector fields $Y_1,\dots,Y_m,Z_1,\dots,Z_k$
on $U$ such that $Y_1,\dots,Y_m$ are tangent to $L$ and each $\varepsilon>0$
\begin{equation}
  \label{e:lps}
\sup_{x,\xi}|Y_1\ldots Y_mZ_1\ldots Z_k (a_1\cdots a_N)|=\mathcal O(h^{-\rho k-\rho'm-\varepsilon}).
\end{equation}
By the Leibniz rule, $Y_1\ldots Y_m Z_1\ldots Z_k (a_1\cdots a_N)$ is a sum
of $N^{m+k}=\mathcal O(h^{-\varepsilon})$ terms.
Each of these summands is a product of $N$ terms, of which at least $N-m-k$
have the form $a_j$ for some $j$, and the rest are obtained
by differentiating~$a_j$. Since the $S^{\comp}_{L,\rho,\rho'}(U)$ seminorms of $a_j$
are bounded uniformly in $j$, each summand is $\mathcal O(h^{-\rho k-\rho'm})$,
giving~\eqref{e:lps}.
\end{proof}

\subsection{Model calculus}
  \label{s:appendix-basic}

In~\S\S\ref{s:appendix-basic}--\ref{s:appendix-general}
we review the construction of the calculus in~\cite[\S\S3.2,3.3]{hgap},
explaining how to modify it to quantize symbols in $S^{\comp}_{L,\rho,\rho'}(U)$.

Following~\cite[\S3.2]{hgap}, we first consider the model case when $M=\mathbb R^n$,
$U=T^*\mathbb R^n$, and $L=L_0$ is the vertical foliation:
$$
L_0=\Span(\partial_{\eta_1},\dots,\partial_{\eta_n}),
$$
where $(y,\eta)$ are the standard coordinates on $T^*\mathbb R^n$.
Symbols in $S^{\comp}_{L_0,\rho,\rho'}(T^*\mathbb R^n)$ satisfy the derivative bounds
\begin{equation}
  \label{e:derby0}
\sup_{y,\eta}|\partial^\alpha_y \partial^\beta_\eta a(y,\eta;h)|\leq C_{\alpha\beta}
h^{-\rho|\alpha|-\rho'|\beta|}.
\end{equation}
For these symbols we use the standard quantization,
\begin{equation}
  \label{e:op-h}
\Op_h(a)f(y)=(2\pi h)^{-n}\int_{\mathbb R^{2n}}e^{{i\over h}(y-y')\cdot \eta}a(y,\eta)f(y')\,dy'd\eta.
\end{equation}
Other quantizations such as the Weyl quantization are likely to produce the same class
of operators, however the standard quantization is convenient for proving invariance
under conjugation by Fourier integral operators, see~\cite[Lemma~3.10]{hgap}.

The standard quantization has the following properties:
\begin{enumerate}
\item for $a\in S^{\comp}_{L_0,\rho,\rho'}(T^*\mathbb R^n)$, the operator
$\Op_h(a):L^2(\mathbb R^n)\to L^2(\mathbb R^n)$ is bounded uniformly in $h$;
\item for $a,b\in S^{\comp}_{L_0,\rho,\rho'}(T^*\mathbb R^n)$, we have
for some~$a\#b\in S^{\comp}_{L_0,\rho,\rho'}(T^*\mathbb R^n)$
\begin{align}
  \label{e:beepy-1}
\Op_h(a)\Op_h(b)&=\Op_h(a\#b)+\mathcal O(h^\infty)_{L^2\to L^2},\\
  \label{e:beepy-2}
a\#b&=ab+\mathcal O(h^{1-\rho-\rho'})_{S^{\comp}_{L_0,\rho,\rho'}(T^*\mathbb R^n)},\\
  \label{e:beepy-2.5}
\supp(a\# b)&\subset \supp a\cap\supp b;
\end{align}
\item for $a\in S^{\comp}_{L_0,\rho,\rho'}(T^*\mathbb R^n)$, we have
for some $a^*\in S^{\comp}_{L_0,\rho,\rho'}(T^*\mathbb R^n)$
\begin{align}
  \label{e:beepy-3}
\Op_h(a)^*&=\Op_h(a^*)+\mathcal O(h^\infty)_{L^2\to L^2},\\
  \label{e:beepy-3.25}
a^*&=\overline a+\mathcal O(h^{1-\rho-\rho'})_{S^{\comp}_{L_0,\rho,\rho'}(T^*\mathbb R^n)},\\
  \label{e:beepy-3.5}
\supp a^*&\subset \supp a;
\end{align}
\item if one of the symbols $a,b$ lies in $S^{\comp}_{L_0,\rho,\rho'}(T^*\mathbb R^n)$
and the other one has all derivatives bounded uniformly in $h$
(it does not have to be compactly supported), then~\eqref{e:beepy-1},
\eqref{e:beepy-2.5} hold and
\begin{equation}
  \label{e:beepy-4}
a\# b=ab+\mathcal O(h^{1-\rho})_{S^{\comp}_{L_0,\rho,\rho'}(T^*\mathbb R^n)};
\end{equation}
\item if $a=a_0+ha_1$ where
$a_0,a_1$ have all derivatives bounded uniformly in $h$, 
$b=b_0+h^{1-\rho}b_1$ where $b_0,b_1\in S^{\comp}_{L_0,\rho,\rho'}(T^*\mathbb R^n)$,
and $\partial_\eta a_0=0$ near $\supp b_0\cup\supp b_1$,
then
\begin{equation}
  \label{e:beepy-5}
a\# b-b\#a=-ih\{a_0,b_0\}+\mathcal O(h^{2-\rho-\rho'})_{S^{\comp}_{L_0,\rho,\rho'}(T^*\mathbb R^n)}.\end{equation}
\end{enumerate}
The proofs are similar to those of~\cite[Lemmas~3.7, 3.8]{hgap}. More precisely,
we use the unitary rescaling operator
$$
T_{\rho,\rho'}:L^2(\mathbb R^n)\to L^2(\mathbb R^n),\quad
T_{\rho,\rho'}u(y)=h^{(\rho-\rho')n\over 4}u(h^{\rho-\rho'\over 2}y),
$$
to conjugate $\Op_h(a)$ as follows:
$$
T_{\rho,\rho'}\Op_h(a)T_{\rho,\rho'}^{-1}=\Op_h(a_{\rho,\rho'}),\quad
a_{\rho,\rho'}(y,\eta;h):=a(h^{\rho-\rho'\over 2}y,h^{\rho'-\rho\over 2}\eta;h).
$$
If $a$ satisfies~\eqref{e:derby0}, then the rescaled symbol $a_{\rho,\rho'}$ satisfies
$$
\sup_{y,\eta}|\partial^\alpha_y \partial^\beta_\eta a_{\rho,\rho'}(y,\eta;h)|\leq C_{\alpha\beta}
h^{-{\rho+\rho'\over 2}(|\alpha|+|\beta|)},
$$
that is $a_{\rho,\rho'}\in S_{\rho+\rho'\over 2}$ where the classes $S_\delta$,
$0\leq\delta\leq 1/2$, are defined in~\cite[(4.4.5)]{e-z}. Then the statements~(1)--(3) above
follow from the standard properties of the $S_\delta$ calculus, see~\cite[Theorems~4.23(ii),
4.14, and~4.17]{e-z}. The statements~(4)--(5) follow by an examination of the terms
in the asymptotic expansion for $a\# b$.

The model calculus satisfies the following version of sharp G\r arding inequality:
\begin{lemm}
  \label{l:garding-model}
Assume that $a\in S^{\comp}_{L_0,\rho,\rho'}(T^*\mathbb R^n)$ satisfies
$\Re a\geq 0$ everywhere. Then there exists a constant $C$ depending on $a$
such that for all $h$ and all $u\in L^2(\mathbb R^n)$
\begin{equation}
  \label{e:garding-model}
\Re\langle \Op_h(a)u, u\rangle_{L^2}\geq -Ch^{1-\rho-\rho'}\|u\|_{L^2}^2.
\end{equation}
\end{lemm}
\begin{proof}
We take the following rescaled versions of $u$, $a$, and $h$:
$$
\tilde u(y):=h^{\rho n/2}u(h^\rho y),\quad
\tilde a(y,\eta;h):=a(h^\rho y, h^{\rho'} \eta;h),\quad
\tilde h:=h^{1-\rho-\rho'}.
$$
Note that $\|\tilde u\|_{L^2}=\|u\|_{L^2}$.
We have
$$
\Re\langle \Op_h(a)u,u\rangle_{L^2}=\Re \langle \Op_{\tilde h}(\tilde a)\tilde u,\tilde u\rangle_{L^2}.
$$
Now~\eqref{e:derby0} implies that all derivatives of $\tilde a$ are bounded uniformly in $h$.
It remains to apply the standard sharp G\r arding inequality~\cite[Theorem~4.32]{e-z}.
\end{proof}

\subsection{Fourier integral operators}
\label{s:fios}

To pass from the model case to the general case, we study the conjugation
of operators in the model calculus by Fourier integral operators.
We briefly review the notation for Fourier integral operators,
referring the reader to~\cite[\S2.2]{hgap} and the references there
for details:
\begin{itemize}
\item Let $M_1,M_2$ be manifolds of the same dimension.
An \emph{exact symplectomorphism} is a diffeomorphism $\varkappa:U_2\to U_1$,
where $U_j\subset T^* M_j$ are open sets, such that
$\varkappa^*(\xi\,dx)-\eta\,dy$ is an exact 1-form.
Here $\xi\,dx$ and $\eta\,dy$ are the canonical 1-forms
on $T^*M_1$ and $T^*M_2$ respectively.
We fix an antiderivative for $\varkappa^*(\xi\,dx)-\eta\,dy$.
\item For an exact symplectomorphism $\varkappa$ (with a fixed antiderivative),
denote by $I^{\comp}_h(\varkappa)$ the class of compactly supported%
\footnote{An operator is called compactly supported if its Schwartz kernel
is compactly supported.}
and compactly microlocalized semiclassical Fourier integral operators
associated to $\varkappa$. These operators
are bounded $L^2(M_2)\to L^2(M_1)$ uniformly in $h$.
\item Let $\varkappa:U_2\to U_1$ be an exact symplectomorphism
and $\varkappa^{-1}$ denote its inverse. For any $B\in I^{\comp}_h(\varkappa)$,
$B'\in I^{\comp}_h(\varkappa^{-1})$, the operators
$BB'$ and $B'B$ are pseudodifferential in the class $\Psi^{\comp}_h$
and~\cite[(2.12)]{hgap}
\begin{equation}
  \label{e:egorov-fio}
\sigma_h(B'B)=\sigma_h(BB')\circ\varkappa.
\end{equation}
If $V_1\subset U_1$, $V_2\subset U_2$ are compact sets such that
$\varkappa(V_2)=V_1$, then we say that $B,B'$ \emph{quantize $\varkappa$
near $V_1\times V_2$} if
\begin{equation}
  \label{e:fio-quantize}
\begin{aligned}
BB'&=I+\mathcal O(h^\infty)\quad\text{microlocally near }V_1,\\
B'B&=I+\mathcal O(h^\infty)\quad\text{microlocally near }V_2.
\end{aligned}
\end{equation}
\end{itemize}
The quantization studied in~\S\ref{s:appendix-basic} is invariant
under conjugation by Fourier integral operators whose underlying symplectomorphisms preserve $L_0$:
\begin{lemm}
  \label{l:gauge-invariance}
Assume that $\varkappa:U_2\to U_1$, $U_j\subset T^*\mathbb R^n$,
is an exact symplectomorphism such that $\varkappa_*(L_0)=L_0$
and take $B\in I^{\comp}_h(\varkappa)$, $B'\in I^{\comp}_h(\varkappa^{-1})$.
Then for each $a\in S^{\comp}_{L_0,\rho,\rho'}(T^*\mathbb R^n)$
there exists $b\in S^{\comp}_{L_0,\rho,\rho'}(T^*\mathbb R^n)$
such that
\begin{align}
  \label{e:gijoe-1}
B'\Op_h(a)B&=\Op_h(b)+\mathcal O(h^\infty)_{L^2\to L^2},\\
  \label{e:gijoe-2}
b&=(a\circ\varkappa)\sigma_h(B'B)+\mathcal O(h^{1-\rho})_{S^{\comp}_{L_0,\rho,\rho'}(T^*\mathbb R^n)},\\
  \label{e:gijoe-3}
\supp b&\subset \varkappa^{-1}(\supp a).
\end{align}
\end{lemm}
\begin{proof}
We argue exactly as in the proofs of~\cite[Lemmas~3.9, 3.10]{hgap}.
The stationary phase asymptotic at the end of the proof of~\cite[Lemma~3.9]{hgap}
produces a remainder $\mathcal O(h^{1-2\rho'})$. The
multiplication formula~\eqref{e:beepy-4} applied to the
expression $A'\Op_h(\tilde a)A$ 
in the last paragraph of the proof of~\cite[Lemma~3.10]{hgap}
gives a remainder $\mathcal O(h^{1-\rho})$.
\end{proof}

\subsection{General calculus}
  \label{s:appendix-general}

We now construct a quantization $\Op_h^L(a)$ of a symbol
$a\in S^{\comp}_{L,\rho,\rho'}(U)$ for a general Lagrangian foliation $L$
on $U\subset T^*M$. This is done similarly to~\cite[\S3.3]{hgap}
by summing operators in the model calculus conjugated by appropriately
chosen Fourier integral operators.

We say that $(U',\varkappa,B,B')$ is a \emph{chart} for $L$
if:
\begin{itemize}
\item $U'\subset U$ is an open set
and $\varkappa:U'\to T^*\mathbb R^n$ is an exact symplectomorphism onto its
image which maps $L$ to $L_0$;
\item $B\in I^{\comp}_h(\varkappa)$ and $B'\in I^{\comp}_h(\varkappa^{-1})$.
\end{itemize}
For each $(x_0,\xi_0)\in U$ there exists a chart $(U',\varkappa,B,B')$ such that
$\sigma_h(B'B)(x_0,\xi_0)\neq 0$ (in fact, we may take $B'=B^*$). Here the existence of $\varkappa$
follows from~\cite[Lemma~3.6]{hgap} and the existence of $B,B'$ is discussed
in the paragraph following~\cite[(2.12)]{hgap}.

Following~\cite[(3.11)]{hgap} we put for $a\in S^{\comp}_{L,\rho,\rho'}(U)$
$$
\Op_h^L(a):=\sum_\ell B_\ell' \Op_h(a_\ell)B_\ell,\quad
a_\ell=(\chi_\ell a)\circ\varkappa_\ell^{-1}\in S^{\comp}_{L_0,\rho,\rho'}(T^*\mathbb R^n),
$$
where $(U_\ell,\varkappa_\ell,B_\ell,B'_\ell)$ is a collection of charts for $L$
such that $U_\ell\subset U$ form a locally finite cover of $U$,
the symbols $\sigma_h(B'_\ell B_\ell)\in C_0^\infty(U_\ell)$
form a partition of unity on $U$, $\chi_\ell\in C_0^\infty(U_\ell)$ are
equal to 1 near $\supp\sigma_h(B_\ell'B_\ell)$,
and $\Op_h$ is defined in~\eqref{e:op-h}.
The quantization procedure $\Op_h^L$ depends on the choice of charts,
however properties~(3)--(4) below show that the resulting class
of operators is invariant.

To simplify the proof of Lemma~\ref{l:gaarding} below, we additionally
assume that $B'_\ell=B_\ell^*$. This can be arranged as follows: note that
for any choice of $\varkappa$ and $B\in I^{\comp}_h(\varkappa)$, we have
$B^*\in I^{\comp}_h(\varkappa^{-1})$
and $\sigma_h(B^*B)\geq 0$ (since $B^*B$ is a pseudodifferential operator which
is nonnegative on $L^2$).
Choose a collection of charts $(U_\ell,\varkappa_\ell,\widetilde B_\ell,\widetilde B^*_\ell)$ for $L$
such that $b:=\sum_\ell \sigma_h(\widetilde B_\ell^*\widetilde B_\ell)>0$ on $U$. Putting
$B_\ell:=\widetilde B_\ell Y_\ell$ where $Y_\ell\in\Psi^{\comp}_h(M)$ satisfy
$\sigma_h(Y_\ell)=b^{-1/2}$ near $\WFh(\widetilde B_\ell^*\widetilde B_\ell)$,
we obtain $\sum_\ell \sigma_h(B_\ell^*B_\ell)=1$ on $U$.

For a compactly supported
operator $A:L^2(M)\to L^2(M)$, we say that $A\in\Psi^{\comp}_{h,L,\rho,\rho'}(U)$
if $A=\Op_h^L(a)+\mathcal O(h^\infty)_{L^2\to L^2}$ for some
$a\in S^{\comp}_{L,\rho,\rho'}(U)$. The quantization
procedure $\Op_h^L$ has the following properties
which are consequences of the results of~\S\ref{s:appendix-basic}--\S\ref{s:fios},
see~\cite[Lemmas~3.12 and~3.14]{hgap}:
\begin{enumerate}
\item For each $a\in S^{\comp}_{L,\rho,\rho'}(U)$, the operator
$\Op_h^L(a):L^2(M)\to L^2(M)$ is compactly supported and
bounded uniformly in $h$.
\item If $a\in C_0^\infty(U)$ is $h$-independent,
then
$\Op_h^L(a)\in \Psi^{\comp}_h(M)$ and
$\sigma_h(\Op_h^L(a))=a$.
\item For each $a\in S^{\comp}_{L,\rho,\rho'}(U)$ and
a chart $(U',\varkappa,B,B')$ for $L$ there exists
$b\in S^{\comp}_{L_0,\rho,\rho'}(T^*\mathbb R^n)$ such that
\begin{equation}
  \label{e:genius-1}
\begin{aligned}
B\Op_h^L(a)B'&=\Op_h(b)+\mathcal O(h^\infty)_{L^2\to L^2},\\
b&=(a\circ \varkappa^{-1})\sigma_h(BB')+\mathcal O(h^{1-\rho})_{S^{\comp}_{L_0,\rho,\rho'}(T^*\mathbb R^n)},\\
\supp b&\subset \varkappa(\supp a).
\end{aligned}
\end{equation}
\item For each $b\in S^{\comp}_{L_0,\rho,\rho'}(T^*\mathbb R^n)$ and
a chart $(U',\varkappa,B,B')$ for $L$ there exists
$a\in S^{\comp}_{L,\rho,\rho'}(U)$ such that
\begin{equation}
  \label{e:genius-2}
\begin{aligned}
B'\Op_h(b)B&=\Op_h^L(a)+\mathcal O(h^\infty)_{L^2\to L^2},\\
a&=(b\circ\varkappa)\sigma_h(B'B)+\mathcal O(h^{1-\rho})_{S^{\comp}_{L_0,\rho,\rho'}(T^*\mathbb R^n)},\\
\supp a&\subset\varkappa^{-1}(\supp b).
\end{aligned}
\end{equation}
\item Assume that $M_1,M_2$ are manifolds of the same dimension,
$U_j\subset T^*M_j$ are open sets,
$L_j$ are Lagrangian foliations on $U_j$,
$U'_j\subset U_j$ are open,
$\varkappa:U'_2\to U'_1$ is an exact symplectomorphism mapping
$L_2$ to $L_1$,
and $B\in I^{\comp}_h(\varkappa)$, $B'\in I^{\comp}_h(\varkappa^{-1})$.
Then for each $a_1\in S^{\comp}_{L_1,\rho,\rho'}(U_1)$ there exists
$a_2\in S^{\comp}_{L_2,\rho,\rho'}(U_2)$ such that
\begin{equation}
  \label{e:genius-3}
\begin{aligned}
B'\Op_h^{L_1}(a_1)B&=\Op_h^{L_2}(a_2)+\mathcal O(h^\infty)_{L^2\to L^2},\\
a_2&=(a_1\circ\varkappa)\sigma_h(B'B)+\mathcal O(h^{1-\rho})_{S^{\comp}_{L_0,\rho,\rho'}(U_2)},\\
\supp a_2&\subset \varkappa^{-1}(\supp a_1).
\end{aligned}
\end{equation}
\item For each $a,b\in S^{\comp}_{L,\rho,\rho'}(U)$ there exists
$a\#_L b\in S^{\comp}_{L,\rho,\rho'}(U)$ such that
\begin{equation}
  \label{e:genius-4}
\begin{aligned}
\Op_h^L(a)\Op_h^L(b)&=\Op_h^L(a\#_L b)+\mathcal O(h^\infty)_{L^2\to L^2},\\
a\#_L b&=ab+\mathcal O(h^{1-\rho-\rho'})_{S^{\comp}_{L,\rho,\rho'}(U)},\\
\supp (a\#_L b)&\subset \supp a\cap \supp b.
\end{aligned}
\end{equation}
\item For each $a\in S^{\comp}_{L,\rho,\rho'}(U)$
there exists $a^*_L\in S^{\comp}_{L,\rho,\rho'}(U)$ such that
\begin{equation}
  \label{e:genius-5}
\begin{aligned}
\Op_h^L(a)^*&=\Op_h^L(a^*_L)+\mathcal O(h^\infty)_{L^2\to L^2},\\
a^*_L&=\overline a+\mathcal O(h^{1-\rho-\rho'})_{S^{\comp}_{L,\rho,\rho'}(U)},\\
\supp a^*_L&\subset \supp a.
\end{aligned}
\end{equation}
\end{enumerate}
The following version of sharp G\r arding inequality follows immediately
from Lemma~\ref{l:garding-model} and the fact that
$B'_\ell=B_\ell^*$:
\begin{lemm}
  \label{l:gaarding}
Assume that $M$ is compact,
$a\in S^{\comp}_{L,\rho,\rho'}(U)$, and $\Re a\geq 0$. Then
there exists a constant $C$ depending on some $S^{\comp}_{L,\rho,\rho'}$ seminorm of $a$ such that
\begin{equation}
  \label{e:gaarding}
\Re\langle \Op_h^L(a) u,u\rangle_{L^2}\geq -Ch^{1-\rho-\rho'}\|u\|_{L^2}^2\quad\text{for all }
u\in L^2(M).
\end{equation}
\end{lemm}
Lemma~\ref{l:gaarding} implies a more precise bound
on the operator norm of $\Op_h^L(a)$:
\begin{lemm}
  \label{l:precise-norm}
Assume that $M$ is compact, $a\in S^{\comp}_{L,\rho,\rho'}(U)$,
and $\sup |a|\leq 1$.
Then there exists a constant $C$ depending on some $S^{\comp}_{L,\rho,\rho'}$ seminorm of $a$ such that
\begin{equation}
  \label{e:precise-norm}
\|\Op_h^L(a)\|_{L^2\to L^2}\leq 1+Ch^{1-\rho-\rho'}.
\end{equation}
\end{lemm}
\begin{proof}
Fix $h$-independent $\chi\in C_0^\infty(U;[0,1])$ such that $\chi=1$ near $\supp a$.
Put
$b:=\chi^2-|a|^2$. Then $b\in S^{\comp}_{L,\rho,\rho'}(U)$ and
$b\geq 0$.
Applying Lemma~\ref{l:gaarding} to $b$, we get
for all $u\in L^2(M)$
$$
\begin{aligned}
\|\Op_h^L(\chi)u\|_{L^2}^2
-\|\Op_h^L(a)u\|_{L^2}^2
&\geq \Re\langle \Op_h^L(b)u,u\rangle_{L^2}
-Ch^{1-\rho-\rho'}\|u\|_{L^2}^2
\\&\geq -Ch^{1-\rho-\rho'}\|u\|_{L^2}^2.
\end{aligned}
$$
Estimating the norm of $\Op_h^L(\chi)$ by~\eqref{e:basic-norm-bound}
we get
$$
\begin{aligned}
\|\Op_h^L(a)u\|_{L^2}^2&\leq \|\Op_h^L(\chi)u\|_{L^2}^2
+Ch^{1-\rho-\rho'}\|u\|_{L^2}^2
\\&\leq \|u\|_{L^2}^2+Ch^{1-\rho-\rho'}\|u\|_{L^2}^2
\end{aligned}
$$
finishing the proof.
\end{proof}
Using Lemma~\ref{l:precise-norm} we get the following operator version of Lemma~\ref{l:long-product-symbols}:
\begin{lemm}
  \label{l:long-product-operators}
Let $a_1,\dots,a_N\in S^{\comp}_{L,\rho,\rho'}(U)$ be as in Lemma~\ref{l:long-product-symbols}
and assume that $A_1,\dots,A_N$ are operators on $L^2(M)$ such that
$A_j=\Op_h^L(a_j)+\mathcal O(h^{1-\rho-\rho'-})_{L^2\to L^2}$ where the constants in $\mathcal O(\bullet)$
are independent of $j$. Then
$$
A_1\cdots A_N=\Op_h^L(a_1\cdots a_N)+\mathcal O(h^{1-\rho-\rho'-})_{L^2\to L^2}.
$$
\end{lemm}
\begin{proof}
We have
$$
\begin{gathered}
A_1\cdots A_N-\Op_h^L(a_1\cdots a_N)
=\sum_{j=1}^N B_j A_{j+1}\cdots A_N,\\
B_j:=
\begin{cases}
A_1-\Op_h(a_1),& j=1;\\
\Op_h^L(a_1\cdots a_{j-1})A_j-\Op_h^L(a_1\cdots a_j),& 2\leq j\leq N. 
\end{cases}
\end{gathered}
$$
Here $\Op_h^L(a_1\cdots a_{j-1})$ is well-defined since by Lemma~\ref{l:long-product-symbols},
$a_1\cdots a_{j-1}\in S^{\comp}_{L,\rho+\varepsilon,\rho'+\varepsilon}(U)$
uniformly in $j$ for any small $\varepsilon>0$.

Since $\sup|a_j|\leq 1$, by Lemma~\ref{l:precise-norm} we have
for some $C$ independent of~$j$
$$
\|A_j\|_{L^2\to L^2}\leq 1+Ch^{1-\rho-\rho'}.
$$
Since $N=\mathcal O(\log(1/h))$, we have uniformly in $j$
$$
\|A_{j+1}\cdots A_N\|_{L^2\to L^2}\leq C.
$$
Therefore it suffices to show that we have uniformly in $j$,
\begin{equation}
  \label{e:lpoint-1}
\|B_j\|_{L^2\to L^2}=\mathcal O(h^{1-\rho-\rho'-})_{L^2\to L^2}.
\end{equation}
For $j=1$ this is immediate so we assume $2\leq j\leq N$.
We may replace $A_j$ by $\Op_h^L(a_j)$ in the definition of $B_j$.
Then~\eqref{e:lpoint-1} follows from the product formula~\eqref{e:genius-4} on the space $S_{L,\rho+\varepsilon,\rho'+\varepsilon}^{\comp}$.
\end{proof}

\subsection{Egorov's theorem}
\label{s:egorov}

We finally prove two versions of Egorov's theorem for the $\Psi^{\comp}_{h,L,\rho,\rho'}(U)$
calculus. In this subsection we assume that $M$
is a compact manifold, $U\subset T^*M$ is open,
$L$ is a Lagrangian foliation on $U$, and $P\in\Psi^{\comp}_h(M)$ is
self-adjoint with principal symbol $p=\sigma_h(P)\in C_0^\infty(T^*M;\mathbb R)$.
We moreover assume that
\begin{equation}
  \label{e:egorov-assume}
L_{(x,\xi)}\subset \ker dp(x,\xi)\quad\text{for all }(x,\xi)\in U;
\end{equation}
this is equivalent to the Hamiltonian vector field $H_p$
lying inside $L$.
The operator $e^{-itP/h}:L^2(M)\to L^2(M)$ is unitary.

We start with the following fixed time statement
similar to~\cite[Lemma~3.17]{hgap}:
\begin{lemm}
  \label{l:egorov-fixed}
Let $a\in S^{\comp}_{L,\rho,\rho'}(U)$ and fix an $h$-independent constant
$T\geq 0$ such that
$e^{-tH_p}(\supp a)\subset U$ for all $t\in [0,T]$.
Then
\begin{equation}
  \label{e:effy-1}
e^{itP/h}\Op_h^L(a)e^{-itP/h}=\Op_h^L(a\circ e^{tH_p})+\mathcal O(h^{1-\rho-\rho'})_{L^2\to L^2}
\quad\text{for }
0\leq t\leq T.
\end{equation}
\end{lemm}
\begin{proof}
We first claim that for each $b\in S^{\comp}_{L,\rho,\rho'}(U)$
\begin{equation}
  \label{e:egorov-commutator}
[P,\Op_h^L(b)]=-ih\Op_h^L(H_p b)+\mathcal O(h^{2-\rho-\rho'})_{L^2\to L^2}.
\end{equation}
Using a partition of unity for $b$ we may assume that there exists a chart $(U',\varkappa,B,B')$
for $L$
such that $B,B'$ quantize $\varkappa$ near $\varkappa(\supp b)\times \supp b$
in the sense of~\eqref{e:fio-quantize}.
Then $B'B=I+\mathcal O(h^\infty)$ microlocally near $\supp b$,
$\sigma_h(B'B)=1$ near $\supp b$,
and $\sigma_h(BB')=1$ near $\varkappa(\supp b)$.
Since both $\Op_h^L(b)$ and $P$ are pseudolocal, we have
$$
\begin{aligned}\relax
[P,\Op_h^L(b)]&=
B'B(PB'B\Op_h^L(b)-\Op_h^L(b)B'BP)B'B+\mathcal O(h^\infty)_{L^2\to L^2}
\\&=B'[BPB',B\Op_h^L(b)B']B+\mathcal O(h^\infty)_{L^2\to L^2}.
\end{aligned}
$$
By~\eqref{e:genius-1} we have
$$
\begin{aligned}
B\Op_h^L(b)B'=\Op_h(\tilde b)+\mathcal O(h^\infty)_{L^2\to L^2}&\quad\text{for some }\tilde b\in S^{\comp}_{L_0,\rho,\rho'}(T^*\mathbb R^n),\\
\tilde b=b\circ\varkappa^{-1}+\mathcal O(h^{1-\rho})_{S^{\comp}_{L_0,\rho,\rho'}(T^*\mathbb R^n)},&\quad
\supp \tilde b\subset \varkappa(\supp b).
\end{aligned}
$$
Next, $BPB'\in\Psi^{\comp}_h(\mathbb R^n)$ and 
by~\eqref{e:egorov-fio},
$\sigma_h(BPB')=(p\circ\varkappa^{-1})\sigma_h(BB')$
is equal to $p\circ\varkappa^{-1}$ near $\supp \tilde b$.
By~\eqref{e:egorov-assume} we then have $\partial_\eta \sigma_h(BPB')=0$
near $\supp\tilde b$.
By~\eqref{e:beepy-5}
\begin{equation}
  \label{e:egorov-comm2}
\begin{aligned}\relax
[P,\Op_h^L(b)]&=B'[BPB',\Op_h(\tilde b)]B+\mathcal O(h^\infty)_{L^2\to L^2}\\
&=-ihB'\Op_h\big(\{p\circ\varkappa^{-1},b\circ\varkappa^{-1}\}\big)B
+\mathcal O(h^{2-\rho-\rho'})_{L^2\to L^2}.
\end{aligned}
\end{equation}
We have $\{p\circ\varkappa^{-1},b\circ\varkappa^{-1}\}=(H_p b)\circ\varkappa^{-1}
\in h^{-\rho'}S^{\comp}_{L_0,\rho,\rho'}(T^*\mathbb R^n)$.
Therefore by~\eqref{e:genius-2}
the right-hand side of~\eqref{e:egorov-comm2}
is equal to $-ih\Op_h^L(H_p b)+\mathcal O(h^{2-\rho-\rho'})_{L^2\to L^2}$,
finishing the proof of~\eqref{e:egorov-commutator}.

Now, put $a_t:=a\circ e^{tH_p}$, $t\in [0,T]$.
By~\eqref{e:egorov-assume} the map
$e^{tH_p}$ preserves the foliation~$L$ on~$\supp a$, therefore
$a_t\in S^{\comp}_{L,\rho,\rho'}(U)$.
Since $\partial_t a_t=H_p a_t$,
by~\eqref{e:egorov-commutator} we have
$$
\begin{aligned}
ih\partial_t (e^{-itP/h}\Op_h^L(a_t)e^{itP/h})&=
e^{-itP/h}\big(ih\Op_h^L(\partial_t a_t)+[P,\Op_h^L(a_t)]\big)e^{itP/h}\\
&=\mathcal O(h^{2-\rho-\rho'})_{L^2\to L^2},\quad
0\leq t\leq T.
\end{aligned}
$$
Integrating this from $0$ to $t$, we get~\eqref{e:effy-1}, finishing the proof.
\end{proof}
We now restrict ourselves to the case when $M$ is a hyperbolic surface,
$U=T^*M\setminus 0$, and $L\in \{L_u,L_s\}$ with $L_u,L_s$
defined in~\eqref{e:l-foliations}.
Let $\varphi_t$ be the homogeneous geodesic flow,
$P\in\Psi^{\comp}_h(M)$ be defined in~\eqref{e:the-P},
and $U(t)=e^{-itP/h}$ as in~\eqref{e:U-t}.
The following statement is a version of Egorov's theorem
for times up to $\rho\log(1/h)$ assuming that the propagated
operator lies in the standard calculus $\Psi^{\comp}_h$:
\begin{lemm}
  \label{l:egorov-long}
Assume that $a\in C_0^\infty(\{1/4<|\xi|_g<4\})$ is $h$-independent.
Then we have uniformly in $t\in [0,\rho\log(1/h)]$
\begin{align}
  \label{e:elliot-1}
U(-t)\Op_h(a)U(t)&=\Op_h^{L_s}(a\circ\varphi_t)+\mathcal O(h^{1-\rho}\log(1/h))_{L^2\to L^2},\\
  \label{e:elliot-2}
U(t)\Op_h(a)U(-t)&=\Op_h^{L_u}(a\circ\varphi_{-t})+\mathcal O(h^{1-\rho}\log(1/h))_{L^2\to L^2}.
\end{align}
Here $a\circ\varphi_t\in S^{\comp}_{L_s,\rho,0}(T^*M\setminus 0)$
and $a\circ\varphi_{-t}\in S^{\comp}_{L_u,\rho,0}(T^*M\setminus 0)$
by~\eqref{e:derbound-long-stable}, \eqref{e:derbound-long-unstable}.
\end{lemm}
\begin{proof}
We prove~\eqref{e:elliot-1}, with~\eqref{e:elliot-2} proved similarly
(replacing $P$ by $-P$). By property~(2) in~\S\ref{s:appendix-general} we may replace $\Op_h(a)$ by $\Op_h^{L_s}(a)$
with an $\mathcal O(h)_{L^2\to L^2}$ error.

We write $t=Ns$ where $0\leq s\leq 2$ and $N\in\mathbb N_0$, $N\leq\log(1/h)$.
Then
$$
\begin{gathered}
U(-t)\Op_h^{L_s}(a)U(t)-\Op_h^{L_s}(a\circ\varphi_t)\\
=\sum_{j=0}^{N-1}
U(-js)\big(U(-s)\Op_h^{L_s}(a\circ\varphi_{(N-1-j)s})U(s)-\Op_h^{L_s}(a\circ\varphi_{(N-j)s})\big)U(js).
\end{gathered}
$$
Since $U(js)$ is unitary, it suffices to prove that uniformly in $j=0,\dots,N-1$
\begin{equation}
  \label{e:elliot-3}
U(-s)\Op_h^{L_s}(a\circ\varphi_{(N-1-j)s})U(s)-\Op_h^{L_s}(a\circ\varphi_{(N-j)s})=\mathcal O(h^{1-\rho})_{L^2\to L^2}.
\end{equation}
Now~\eqref{e:elliot-3} follows from Lemma~\ref{l:egorov-fixed} applied
to $a\circ\varphi_{(N-1-j)s}\in S^{\comp}_{L_s,\rho,0}(T^*M\setminus 0)$.
Here $\varphi_t=\exp(tH_{\sigma_h(P)})$ on $\{1/4<|\xi|_g<4\}$
by~\eqref{e:the-P-2}.
\end{proof}

\medskip\noindent\textbf{Acknowledgements.}
The authors would like to thank Jeffrey Galkowski and Maciej Zworski
for several useful discussions and two anonymous referees for numerous
comments used to improve the manuscript.
This research was conducted during the period SD served as
a Clay Research Fellow and LJ as a Visiting Assistant Professor in Purdue University.


\end{document}